\documentclass[article]{amsart}

\usepackage{amsmath, amsthm, amsfonts,stmaryrd,amssymb}
\usepackage[top=1in, bottom=1.25in, left=1.25in, right=1.25in]{geometry}
\usepackage{mathrsfs}    
\usepackage{ dsfont }                  % Nice \mathscr fonts
\usepackage{mathtools}                   % \mathrlap
\usepackage{booktabs}
\usepackage{tikz}
\usepackage{todonotes}
\usepackage[all]{xy}
\usepackage[utf8]{inputenc}
\usepackage{mathtools}                     % Extendable arrows with text
\usepackage{todonotes}
\usepackage{faktor}
\usepackage{ marvosym }
\usepackage{bbm}                         % For \mathbbm{1}
\usepackage{framed}
\usepackage{graphicx}   

\usepackage{hyperref} % crossreferences
\hypersetup{
    colorlinks=true,
    linkcolor=blue,  
    urlcolor=cyan,%,
    pdftitle={WFRtoHDiv},
    pdfauthor={T. GALLOUET, F.X. VIALARD},
} 
\usepackage{theoremref}
\numberwithin{equation}{section}
\theoremstyle{plain}
\newtheorem{remark}{Remark}
\newtheorem{theorem}{Theorem}
\newtheorem{lemma}[theorem]{Lemma}
\newtheorem{corollary}[theorem]{Corollary}
\newtheorem{proposition}[theorem]{Proposition}

\theoremstyle{definition}
\newtheorem{definition}{Definition}

\def\cone{\mathscr{C}}

 \def\m{\mathsf{m}}

\def\R{{\mathbb R}}

\def\N{{\mathbb N}}

\def\exp{\operatorname{exp}}
\def\cexp{\operatorname{c-exp}}
\def\dcone{\d_{\cone(M)}}
\def\ddcone{\d^2_{\cone(M)}}

\def\coneexp{\operatorname{\ddcone-exp}}
\def\Id{\operatorname{Id}}
\def\Dens{\operatorname{Dens}}
\def\Div{\operatorname{div}}

\def\Diff{\operatorname{Diff}}

\let\on=\operatorname

\newcommand{\FF}{F}

\newcommand{\spt}{\on{supp}}
\newcommand{\ud}{\,\mathrm{d}}
\renewcommand{\d}{\,\mathrm{d}}

\def\autgc{ \on{\overline{Aut}}}
\def\isog{ \autgc_{\on{vol}}}
\def\isogh{ \on{\overline{Aut}_{\rho_0,\mu_0}}}
\def\Mes{ \mathcal{M}es}

\newcommand{\eqdef}{\ensuremath{\stackrel{\mbox{\upshape\tiny %def.
}}{:=}}}

\begin{document}
\title[Regularity theory for unbalanced optimal transport]{Regularity theory and geometry of unbalanced optimal transport}
\author{Thomas Gallou\"et} 
\address{INRIA, Project team Mokaplan, Université Paris-Dauphine, PSL Research University, UMR CNRS 7534-Ceremade.}
\email{thomas.gallouet@inria.fr}
\author{Roberta Ghezzi}
\address{Dipartimento di Matematica, Universit\`a degli Studi di Roma ``Tor Vergata'', Rome Italy, orcid ID 0000-0003-2305-7627} 
\email{ghezzi@mat.uniroma2.it}
\author{Fran\c{c}ois-Xavier Vialard}
\address{Université Gustave Eiffel, LIGM, CNRS \\ INRIA, Project team Mokaplan}
\email{fxvialard@normalesup.org}
\maketitle

\begin{abstract}
Using the dual formulation only, we show that the regularity of unbalanced optimal transport also called entropy-transport inherits from the regularity of standard optimal transport. 
We provide detailed examples of Riemannian manifolds and costs for which unbalanced optimal transport is regular.
Among all entropy-transport formulations, Wasserstein-Fisher-Rao (WFR) metric, also called Hellinger-Kantorovich, stands out since it admits a dynamic formulation, which extends the Benamou-Brenier formulation of optimal transport. 
After demonstrating the equivalence between dynamic and static formulations on a closed Riemannian manifold, we prove a polar factorization theorem, similar to the one due to Brenier and Mc-Cann. 
As a byproduct, we formulate the Monge-Ampère equation associated with WFR metric, which also holds for more general costs. Last, we study the link between $c$-convex functions for the cost induced by the WFR metric and the cost on the cone. The main result is that the weak Ma-Trudinger-Wang condition on the cone implies the same condition on the manifold for the cost induced by WFR. 
\end{abstract}

\tableofcontents

\section{Introduction}
In the past few years, optimal transport has seen an impressive development mainly driven by applied fields in which real data require robust and largely applicable models. 
In many applications, data are modeled by probability distributions. 
To compare two such distributions, optimal transport (OT) gives a geometrically meaningful distance. Indeed, OT lifts a distance on the base space to the space of probability measures.
In OT, the underlying idea consists in explaining the variation of mass between measures via displacement, thereby having a global constraint of equal total mass for the two measures.
The last constraint can easily be alleviated with global renormalization but the obtained model will not be able to account for possible local change of mass.
Considering this shortcoming \cite{GeorgiouPower,refId0}, it was natural to enrich the model using local change of mass as proposed by the last author and co-authors and independently by others in \cite{GeneralizedOT1,GeneralizedOT2,new2015kondratyev,LMS}.

When looking for a generalization of optimal transport to unnormalized measures, there are at least two possible directions. The first one consists in extending the Kantorovich formulation of optimal transport, which is static in contrast to the Benamou-Brenier formulation. This idea amounts to relaxing the marginal constraints using some divergence such as the relative entropy (Kullback-Leibler). By doing so, it is not trivial to know whether the resulting functional gives a proper distance between positive densities. The second one is to start by the dynamic formulation of Benamou and Brenier \cite{benamou2000computational}, which is of interest since it uncovers the Riemannian-like structure of the Wasserstein metric for the $L^2$ cost. A natural Riemannian tensor on the space of densities which is one-homogeneous with respect to rescaling of mass is the Hessian of the entropy, known as the Fisher-Rao metric when restricted to the set of probability densities.

The latter idea was the starting point of the concurrent works \cite{GeneralizedOT1,GeneralizedOT2,new2015kondratyev,LMS} that introduced what is now called \emph{unbalanced optimal transport} and which has seen several applications in data sciences. Arguably, the most significant result on this model is the equivalence between the static formulation and the dynamic formulation \cite{GeneralizedOT2,LMS}. Importantly, the article \cite{LMS} gives another characterization of unbalanced optimal transport as a standard optimal transport problem on the cone over the base manifold with second-order moment constraints. This formulation was exploited in  \cite{GALLOUET20184199,AndreaEmbedding} to reformulate the Camassa-Holm equation as a standard incompressible Euler equation on an extension of the cone. Then, generalized flows {\it à la Brenier} were studied in \cite{AndreaRelaxation} for the  Camassa-Holm equation and its higher-dimensional extension. Other interesting extensions and related works of the unbalanced framework include the projection of this distance to the set of probability measures using homogeneity property \cite{laschos2018geometric} and gradient flows that retain more convexity than standard Wasserstein gradient flows \cite{kondratyev2019spherical,kondratyev2019convex}. The dynamic formulation of unbalanced optimal transport has also drawn some interest \cite{bredies2020superposition,bauer2021square}, for defining new metrics between metric measure spaces \cite{deponti2020entropytransport,sejourne2020unbalanced}.
Applications of unbalanced optimal transport are numerous \cite{pmlr-v129-wang20c,shen2021accurate,sjourn2019sinkhorn,sejourne2020unbalanced,Feydy2017}, in particular in data science and computer vision, since this model is more robust in some sense than standard optimal transport and computationally feasible using entropic regularization \cite{chizat2016scaling}.

An open question in this unbalanced framework is the issue of regularity. In the context of standard optimal transport, regularity appeared after Brenier stated the existence of an optimal transport map under mild conditions in Euclidean space. Since then, an ``implicit'' regularity of optimal transport was discovered in \cite{caf92} and following works, see \cite{F-DP} for a recent overview.
Regularity does not hold in general but it is observed when the underlying densities are regular and have convex support in Euclidean space. These results are based on Monge-Ampère equation and they have triggered a number of works concerned with extensions to Riemannian manifolds \cite{McCann2001}, culminating with the identification of the Ma-Trudinger-Wang (MTW) tensor, related to the sectional curvature tensor of a pseudo-Riemannian metric \cite{kim2007continuity}. The nonnegativity of this tensor is a necessary condition for the smoothness of standard optimal transport.

\textbf{Contributions and structure of the article. }
In this paper, we address the question of regularity of unbalanced optimal transport.
We focus on two important instances of the problem that give rise to a metric on the space of positive Radon measures, namely the Wasserstein-Fisher-Rao (or Hellinger-Kantorovich) and the Gaussian-Hellinger distances.
In contrast with standard optimal transport, there is not just a single map that is the solution of the problem. 
However, the objects of interest are still encoded via optimal potentials, on which regularity can be studied. Alternatively, regularity can also be tackled from the primal formulation.  Indeed, a plan that minimizes the primal formulation of unbalanced optimal transport is an optimal transport plan between its marginals. 
\par
From the above remarks, it is natural to expect that the regularity of the potentials is inherited from regularity theory for optimal transport. 
This is actually the case, and we prove this fact, Theorem~\ref{ThReduction},  in Section \ref{secthm4}   by studying the dual formulation and in particular its first-order optimality condition which encodes optimal transport between the optimal marginals of the primal formulation. Starting from the general formulation of \cite{LMS}, our regularity theorem requires Lipschitz regularity of the optimal potentials. The existence of Lipschitz potentials is the main question  of Section  \ref{SecExistenceLipschitz}, where we answer positively,  under geometric conditions on the measures.  We apply our results in Section \ref{sec:leeli} to obtain regularity of unbalanced OT for Gaussian-Hellinger and Wasserstein-Fisher-Rao. In particular, Gaussian-Hellinger is regular on the sphere and the Euclidean space, whereas Wasserstein-Fisher-Rao is regular only on the sphere but not on the Euclidean space.
We then focus in Section \ref{SecStaDyn} on the Wasserstein-Fisher-Rao metric for which we show the equivalence between static and dynamic formulations on a closed Riemannian manifold.  
To derive our main contribution in this section, we take advantage of a geometric point of view to show a polar factorization \cite{Brenier1991, McCann2001} theorem on a semi-direct product of groups, which is the natural extension of the diffeomorphism group to the unbalanced setting. 
Such a decomposition inherits the regularity results of unbalanced optimal transport. 
Our main contribution on the geometric side is presented in Section \ref{SecCconvex} in which we study $c$-convex function for the cost on the cone and the cost induced by the Wasserstein-Fisher-Rao metric. We prove that if the so-called weak MTW condition holds on the cone then it is also the case for the cost induced by Wassertein-Fisher-Rao. Surprisingly, a converse result holds for \emph{cross-curvature} as introduced by Kim and McCann in \cite{kim2007continuity}.

\section{Regularity of unbalanced optimal transport}\label{SecDualRegularity}

\subsection{From optimal transport regularity to unbalanced optimal transport regularity}\label{secthm4}

In what follows, we use the notation $X,Y$ for two spaces that are either Euclidean spaces, bounded convex sets of Euclidean spaces, or Riemannian manifolds. The results in this section apply to the more general setting of \cite{LMS} but since we are interested in regularity theory, we choose to focus on the aforementioned cases.

We consider the general case of an entropy function,  that replaces the relative entropy.   
\begin{definition}
An {\it entropy function} $\FF: \R \to [0,+\infty]$  is a 
 convex, lower semi-continuous, nonnegative function such that $\FF(1) = 0$ and $\FF(x) = +\infty$ if $x<0$. Its {\it recession constant} is defined as $\FF^{'}_{\infty} = \lim_{r\to+\infty} \frac{\FF(r)}{r}$. 
\end{definition}
In the sequel, we denote by $\partial G(x_0)$ the subdifferential of a function $G:\R\to\R$ at a point $x_0$.
\begin{proposition}
The Legendre-Fenchel transform of $\FF$, denoted by $\FF^*$, has a domain of definition $\on{dom}(\FF^*) = (-\infty,\FF^{'}_{\infty}]$ and it satisfies
\begin{equation*}
    \partial \FF^*(\on{dom}(\FF^*)) \subset \R_{\geq 0}\,.
\end{equation*}
Moreover, if $\partial \FF(0) = +\infty$, then $\partial \FF^*(\on{dom}(\FF^*)) \subset \R_{>0}$.
\end{proposition}
\begin{remark}\label{rem_FKL}
The hypothesis  $\partial \FF(0) = +\infty$ is satisfied, for instance, 
for $\FF(x) = x\log(x) - x + 1$, arguably the most important and most frequent entropy function used in unbalanced optimal transport. In this case, the Legendre-Fenchel transform is $\FF^*(x) = e^{x} - 1$.
\end{remark}

\begin{definition}\label{ThCsiszar}
Let $\FF$ be an entropy function and $\mu,\nu$ be Radon measures on a Riemannian manifold $M$. The {\it Csiszàr divergence} associated with $\FF$ is
\begin{equation*}
D_\FF(\mu,\nu) = \int_M \FF\left(\frac{\ud \mu(x)}{\ud \nu(x) }\right) \ud \nu(x) + \FF^{'}_{\infty} \int_M\ud \mu^{\perp}\,,
\end{equation*}
where $\mu^{\perp}$ is the orthogonal part of the Lebesgue decomposition of $\mu$ with respect to $\nu$. 
\end{definition}
For $\FF(x) = x\log(x) - x + 1$, $D_\FF$ is also known as  {\it Kullback-Leibler divergence} or {\it relative entropy}, and it reads
\begin{equation*}%\label{kldiv}
\on{KL}(\mu, \nu) = \int \frac{\ud \mu}{\ud \nu} \log\left(\frac{\ud \mu}{\ud \nu}\right) \ud \nu + |\nu| - |\mu|  \,.
\end{equation*}

Given $\FF$, the resulting divergence $D_\FF$ is jointly convex and lower semi-continuous on the space of pairs of finite and positive Radon measures, see \cite[Corollary 2.9]{LMS}.
We can now define the primal formulation of unbalanced optimal transport, which is similar to the Kantorovich formulation of optimal transport. We denote by $\mathcal{M}_+(X)$ the space of finite and positive Radon measures on $X$.
As is standard in optimal transportation, we need a cost function, i.e., a function $c:M\times M\to\R\cup\{+\infty\}$ that is assumed to be bounded below. Remark that, in our setting, cost functions are allowed to be unbounded above. 

\begin{definition}[Kantorovich UOT]\label{Def:UOTkanto} 
Let $(\rho_0,\rho_1)\in\mathcal{M}_+(X)\times\mathcal{M}_+(Y)$ and $\FF_0,\FF_1$ be entropy functions. The {\it unbalanced optimal transport problem} is defined as
\begin{equation}\label{EqPrimalUOT}
    \on{UOT}(\rho_0,\rho_1) = \inf_{\gamma \in \mathcal{M}_+(X\times Y)} D_{\FF_{0}}(\gamma_0,\rho_0) + D_{\FF_1}(\gamma_1,\rho_1) + \int_{X\times Y} c(x,y) \ud \gamma(x,y)\,,
\end{equation}
where $\gamma_0,\gamma_1$ are marginals of $\gamma$, and $c:X\times Y\to \R\cup \{+\infty\}$ is a cost function (see Definition~\ref{costi} below).
\end{definition}
The distance between two Dirac masses can be computed explicitly. Let $\rho_0 = r\delta_x $, $\rho_1 = s\delta_y$, in order to compute $\on{UOT}(\rho_0,\rho_1)$ one has to compute the local quantity 
\begin{equation*}%\label{EqCostBetweenDiracs}
\mathcal{D}((x,r),(y,s)) \coloneqq \inf_{z \in \R_{>0}} \left(rF_0(z/r) + sF_1(z/s) + c(x,y)z\right)\,.
\end{equation*}
 For the Kullback-Leibler divergence, i.e., when $F_0(x)=F_1(x)=x\log x-x+1$,  this quantity can be computed explicitly as 
\begin{equation*}
    \mathcal{D}((x,r),(y,s)) = r + s - 2e^{-c(x,y)/2}\sqrt{rs}.
\end{equation*}
This will be useful to achieve the Monge formulation of UOT.

The UOT problem has many equivalent formulations. 
In this section, we focus on the dual formulation of \eqref{EqPrimalUOT} given by the Fenchel-Rockafellar theorem.

\begin{proposition}[Dual UOT]\label{prop:formulationduale}
The dual formulation of \eqref{EqPrimalUOT} is 
\begin{equation}\label{EqDualUOTProblem}
    \sup_{(z_0,z_1) \in C_b(X)\times C_b(Y)} -\int_X \FF_0^*(-z_0(x)) \ud \rho_0(x) -\int_Y \FF_1^*(-z_1(y)) \ud \rho_1(y)
\end{equation}
under the constraint \begin{equation}\label{EqInequalityConstraint}
    z_0(x) + z_1(y) \leq c(x,y)\,.
\end{equation}% We denote by 
%\begin{equation*}
%$\mathcal{S}(z_0,z_1)$ %in \eqref{EqDualUOTProblem}.
%=-\int_X \FF_0^*(-z_0(x)) \ud \rho_0(x) -\int_Y \FF_1^*(-z_1(y)) \ud \rho_1(y)
%\end{equation*}
%the functional \eqref{EqDualUOTProblem} above.
\end{proposition}
For the proof in the general case, see for instance \cite[Proposition 4.3]{LMS}.
\par
Our goal is to show that the regularity of unbalanced optimal transport follows from the regularity of standard optimal transport for the cost $c$. This result can be expected since, once the optimal marginals $\gamma_0,\gamma_1$ are fixed in  \eqref{EqPrimalUOT}, optimizing on the plan $\gamma$ (with fixed marginals) is indeed a standard optimal transport problem between $\gamma_0$ and $\gamma_1$ for the cost $c$.

\begin{lemma}[Linearized UOT]\label{ThLinearization}
Assume that the entropy functions $\FF_i$ are differentiable on their domain.
Let $(z_0^\star,z_1^\star)  \in C_b(X) \times C_b(Y)$ be a pair of optimal potentials for the dual problem \eqref{EqDualUOTProblem} satisfying $\on{range}(-z_i^\star) \subset \on{dom}(\FF_i^*)$. 
Then $(z_0^\star,z_1^\star)$ is a solution of the standard optimal transport problem 
\begin{equation}\label{LDOT}
\sup_{(z_0,z_1)\in C_b(X)\times C_b(Y)} \int_X z_0(x) \ud \tilde \rho_0(x) + \int_Y z_1(y) \ud \tilde \rho_1(y) 
\end{equation}
under the constraint $z_0(x) + z_1(y) \leq c(x,y)$ where $\tilde \rho_i =  {\FF_i^*}'(-z_i^\star) \rho_i$ for $i = 0,1$.
\end{lemma}

\begin{proof} Let $(\delta z_0,\delta z_1)\in C_b(X)\times C_b(Y)$ denote a first order admissible variation of $(z_0,z_1) \in C_b(X)\times C_b(Y)$ satisfying the inequality constraint $z_0(x) + \delta z_0(x) + z_1(y) + \delta z_1(y) \leq c(x,y)$. 
Differentiating the dual functional \eqref{EqDualUOTProblem}, we obtain 
\begin{equation*}
    \int_X \delta z_0(x)  {\FF_0^*}'(-z_0(x)) \ud \rho_0(x)  + \int_Y \delta z_1(y) {\FF_1^*}'(-z_1(y)) \ud \rho_1(y)\,.
\end{equation*}
At $(z_0^\star,z_1^\star)$ the optimality implies, for all admissible $(\delta z_0^\star,\delta z_1^\star)$,
\begin{equation*}
    \int_X \delta z_0^\star(x)  {\FF_0^*}'(-z_0^\star) \ud \rho_0(x)  + \int_Y \delta z_1^\star(y) {\FF_1^*}'(-z_1^\star(y)) \ud \rho_1(y)\leq 0,
\end{equation*}
or, equivalently, by linearity, 
\begin{multline}
    \int_X \left(z_0^\star+ \delta z_0^\star(x) \right) {\FF_0^*}'(-z_0^\star) \ud \rho_0(x)  + \int_Y \left(z_1^\star+\delta z_1^\star(y)\right) {\FF_1^*}'(-z_1^\star(y)) \ud \rho_1(y)\leq \\ \label{aide}\int_X  z_0^\star(x)  {\FF_0^*}'(-z_0^\star) \ud \rho_0(x)  + \int_Y  z_1^\star(y) {\FF_1^*}'(-z_1^\star(y)) \ud \rho_1(y),
\end{multline}
for all $(\bar z_0,\bar z_1)=(z_0^\star+ \delta z_0^\star,z_1^\star+ \delta z_1^\star)$ satisfying 
$\bar z_0(x) + \bar z_1(y) \leq c(x,y)$. 
Inequality \eqref{aide} exactly states that $(z_0^\star,z_1^\star)$ is optimal in the constrained problem \eqref{LDOT}.

\end{proof}
\begin{remark}
An immediate consequence of this proof is that the corresponding Radon measures $\tilde \rho_i$ have the same total mass. 
Indeed, given a pair of potentials $(z_0,z_1)$ satisfying \eqref{EqInequalityConstraint}, for every  $\lambda \in \R$ the pair  $(z_0 + \lambda,z_1- \lambda)$ still satisfies  \eqref{EqInequalityConstraint}. 
However, the linearized objective functional differs with the term $\lambda (|\tilde \rho_0| - |\tilde \rho_1|)$ where $|\cdot |$ denotes total mass.
This term can be made arbitrarily large unless $|\tilde \rho_0| = |\tilde \rho_1|$, thus contradicting the fact that the linearization is bounded.
\end{remark}
\begin{remark}\label{regpde} 
Thanks to Lemma~\ref{ThLinearization} and Brenier's work \cite[Section~1.4]{Brenier1991}, given  $(z_0^\star,z_1^\star)$, being optimal for the  problem in Proposition \ref{prop:formulationduale} can be taken as the definition of  variational solutions to a $\on{UOT}$-Monge-Amp\`ere equation given by 
\begin{multline}\label{rregpde}
   \det\left[ -\nabla^2 z_0^\star(x) +(\nabla_{xx}^2 c)(x,(\nabla_{x} c(x,\cdot))^{-1} z_0^\star(x))\right] \\ = \left|\det \left[ (\nabla_{x,y} c)(x,(\nabla_{x} c(x,\cdot))^{-1} z_0^\star(x)) \right] \right|  \frac{ {\FF_0^*}'(-z_0^\star)\rho_0(x)}{{\FF_1^*}'(-z_1^\star(y))\rho_1\circ (\nabla_{x} c(x,\cdot))^{-1} z_0^\star(x)) }\,, 
\end{multline}
where the right-hand side also depends on $z_0^\star$. We detail this computation in the  Gaussian-Hellinger and Hellinger-Kantorovich case in Proposition~\ref{prop:approx} below.
\end{remark}
By Remark~\ref{regpde}, regularity properties for optimal potentials are related to regularity properties of solutions to partial differential equations as in \eqref{rregpde}.
Before stating the main result of this section, we introduce the following definition, which essentially encapsulates the regularity of standard optimal transport needed for its extension to the unbalanced setting. 

\begin{definition}\label{regu}%(Class of $k$-regular pair of measures)
%Let $X, Y $ be two domains and a cost function $c : X \times Y \mapsto \R_+$. 
Let $k\in\N$, $\alpha\in(0,1)$.
Let $(\rho_0,\rho_1)\in \mathcal M_+(X)\times \mathcal M_+(Y)$ be two  measures that are absolutely continuous with respect to some reference volume with %regular 
densities
$(\rho_0,\rho_1)\in  C^{k,\alpha}(X)\times C^{k,\alpha}(Y)$. We say that $(\rho_0,\rho_1)$ is a {\it $k$-regular pair of measures} if, for every $0 \leq l \leq k$ and every pair $(\lambda_0,\lambda_1) \in C^{l,\alpha}(X)\times C^{l,\alpha}(Y)$ of positive functions bounded away from zero and infinity, the optimal potentials, solutions to the standard optimal transport problem between the pair $(\tilde \rho_0,\tilde\rho_1)$, where $\tilde\rho_i = \frac{\lambda_i \rho_i}{|\lambda_i \rho_i|}$, are of class  $C^{l+2,\alpha}$.
\end{definition}
Typical instances for the standard optimal transport problem to be regular are stated in terms of geometric properties on measures' support. This provides a sufficient condition for a pair of measures to be $k$-regular: if both measures have $C^k$ positive densities that are bounded away from zero and infinity, then the pair is $k$-regular if the supports of both measures are convex domains (see   \cite[Theorem 3.3]{philippis2013mongeampre}).  
More generally, Definition~\ref{regu} fits well within the regularity theory developed for Monge-Ampère-type equations. Indeed, geometric assumptions, such as the convexity of the support, which are invariant under pointwise multiplication by a positive function, directly provide $k$-regularity. 

We are in a position to obtain the main result, stating that unbalanced optimal transport inherits the regularity of standard optimal transport associated with the cost $c$.

In the following statement, $F_0,F_1$ and $c$ are, respectively, given entropy functions and cost. We consider the unbalanced optimal transport problem between given measures $\rho_0$ and $\rho_1$ as in \eqref{EqPrimalUOT}.
\begin{theorem}[Reduction to standard optimal transport]\label{ThReduction}
Assume that

(1) the Fenchel-Legendre transform of the entropy functions have domain $[0,+\infty)$, are $C^{k+1}$ on $(0,\infty)$ and $\partial \FF_i(0) = +\infty$, $i=0,1$; 

(2) the pair of measures $(\rho_0,\rho_1)$ is $k$-regular;

(3) the optimal potentials for unbalanced optimal transport $(z_0^\star,z_1^\star)$ are Lipschitz continuous.
Then, the optimal pair $(z_0^\star,z_1^\star)$ is of class $ C^{k+2,\alpha}(X)\times C^{k+2,\alpha}(Y)$.
\end{theorem}

Assumption (1) ensures that the resulting marginals are sufficiently smooth and with unchanged support, i.e., the multiplicative term ${\FF_i^*}'(-z_i^\star)$ does not vanish. 

Assumption (2) presumes that a theory of regularity for a class of optimal maps in the case of standard optimal transport is available. 
This is the case, for instance, under conditions on the Ma-Trudinger-Wang tensor, see e.g.\cite[Chapter 12]{villanioldnew}. 
Relations between the MTW tensor on the underlying space $X$ and the MTW tensor on the cone over $X$ are discussed in Section~\ref{SecCconvex} below. 

Assumption (3) is in general a consequence of Lipschitz continuity of the cost.
However, for unbounded costs, 
existence of Lipschitz potentials for UOT requires more assumptions, as is shown in the next section, where we study the case of Wasserstein-Fisher-Rao metric. (Here we focus on regularity).
 
 The main idea is to use Lemma~\ref{ThLinearization} to relate UOT to standard optimal transport and then start by Assumption (2) to apply a bootstrap argument. 

\begin{proof}
Since the optimal potentials are Lipschitz, Lemma \ref{ThLinearization} gives that these potentials are optimal for a standard optimal transport problem between a new pair of densities which inherits smoothness from the potentials and the initial densities, namely $\tilde \rho_i =  {\FF_i^*}'(-z_i^\star) \rho_i$. 
Hypothesis (1) gives that ${\FF_i^*}'(-z_i^\star)$ is $C^l$ if $z_i \in C^l$ for $l \leq k$. It implies that the regularity of $\tilde \rho_i$ is given by that of $z_i$.
At the initialization step of the bootstrap, they are only Lipschitz, then applying Lemma~\ref{ThLinearization} and Hypothesis (2), the optimal potentials gain in regularity to be $C^{3,1}$.
Then, in turn, we obtain that the marginals $\tilde \rho_i$ are $C^{\min(k,3), 1}$. Iterating this bootstrap argument gives the result, the optimal potentials are $C^{k+2,\alpha}$ and the optimal marginals $\tilde \rho_i$ are $C^{k,\alpha}$. 
\end{proof}

\subsection{Existence of Lipschitz potentials for unbounded costs}\label{SecExistenceLipschitz} 
A natural question about the range of applicability of Theorem~\ref{ThReduction} is to understand whether assumption (3)  is fulfilled by a nonempty class of problems.
In this section, we prove the existence of Lipschitz potentials for the maximization problem in \eqref{EqDualUOTProblem}, \eqref{EqInequalityConstraint} for unbounded costs under an admissibility assumption on the source and target measure.  
Such a condition may be interpreted by saying that pure creation/destruction of mass is forbidden or, in other words, mass transport must be performed between the source and target measure on the whole supports.

For simplicity, we consider the case where $M$ is either a compact Riemannian manifold or a convex and compact domain in Euclidean space.
%Let us recall the notion of conjugate function.
%Let $c:M\times M\to\R \cup \{ + \infty\}$ be a cost function.
%The {\it $c$-conjugate} of a function $z:M \to \R$ is defined by 
%\begin{equation*}
%\hat{z}(x) = \inf_{y \in M} c(x,y) - z(y)\,.
%\end{equation*}
We now define a class of functions that will be considered in this section as costs. 
In particular, such costs can be unbounded.

\begin{definition}\label{costi}
A function $c:M\times M\to\R\cup\{+\infty\}$ is said to be a {\it cost function} if it is bounded below,  it is continuous at every point in $c^{-1}(\R)$, and, for every $L \in \R$, the restriction of $c$ on the sub-level $c^{-1}((-\infty,L])$ is Lipschitz continuous.
\end{definition}
The Lipschitz constant on a sub-level may depend on the chosen $L$. 
\begin{definition}[Admissible measures]\label{defadm}
A pair of   measures $(\rho_0,\rho_1)\in{\mathcal M}_+(M)^2$ is {\it admissible} if, denoting $K_i = \spt\rho_i$, $K_i\neq\emptyset$ $i=0,1$, and there holds
\begin{equation*}
\max\left(\sup_{x \in K_0} \inf_{y \in K_1} c(x,y), \sup_{y \in K_1} \inf_{x \in K_0} c(x,y)\right) < \infty\,.
\end{equation*}
We denote this finite number by $c_H(\rho_0,\rho_1)$. 
\end{definition}
When considering the distance as a cost function, being admissible simply means that the supports of the source and target measure have finite  Hausdorff distance.

The main result of this section provides the existence of Lipschitz solutions to unbalanced optimal transport \eqref{EqDualUOTProblem} between admissible measures for costs functions as in Definition~\ref{costi}.
The admissibility assumption  is crucial in order to overcome the fact that the cost is not bounded above. 
Furthermore, in the framework of a possibly unbounded cost, we are lead to use a ``local'' notion of $c$-conjugate function instead of the usual one, where locality is related to the given pair of measures. 

\begin{proposition}\label{Pr:existence}
Assume $\FF_i$, $i=0,1$, are entropy functions such that $\FF^*_i$, $i=0,1$, is strictly increasing.
Let $(\rho_0,\rho_1)\in\mathcal{M}_+(M)^2$ be a pair of admissible measures. Then there exists an optimal pair 
$(z_0,z_1)\in C(M)^2$ for the maximization problem in \eqref{EqDualUOTProblem}. Moreover, $z_i$ is  Lipschitz continuous on $\spt\rho_i$, $i=0,1$  and 
\begin{eqnarray}
    \forall\, y\in \spt \rho_1,&& z_1(y)=\inf_{x\in\spt{\rho_0}}c(x,y)-z_0(x),\label{eqz1}\\
     \forall\, x\in \spt \rho_0,&&    z_0(x)=\inf_{y\in\spt{\rho_1}}c(x,y)-z_1(y)\label{eqz0}.
\end{eqnarray}
\end{proposition}

Let us first prove an auxiliary technical lemma. 
\begin{lemma}\label{ThLemma}
Let $(\rho_0,\rho_1)$ be an admissible pair of measures. Then, there exist $x_1, \ldots, x_k \in M$ and $r_1,\ldots,r_k >0$ such that $\rho_0(B(x_i,r_i)) > 0$ and for any $y \in \spt\rho_1$, there exists $\bar i\in\{1,\dots, k\}$ such that $\sup_{x \in B(x_{\bar i},r_{\bar i})} c(x,y) < c_H(\rho_0,\rho_1) + 1$.
\end{lemma}
\begin{proof} Set $K_i=\spt\rho_i$, $i=0,1$.
Since the pair $(\rho_0,\rho_1)$ is admissible, for every $y \in K_1$, there exists $B(x_y,r_y)$ and $B(y,\delta_y)$ small enough such that  $\sup_{x_1 \in B(x_y,r_y),y_1\in B(y,\delta_y)}c(x_1,y_1) < c_H(\rho_0,\rho_1) +1$ and $\rho_0(B(x_y,r_y)) >0$. As $K_1$ is compact, there exists a finite number of points $(y_i)_{i=1,\ldots,k}$ such that $K_1\subset\cup_{i=1}^k B(x_y,r_y)$.
Therefore with $x_i=x_{y_i}$ and $r_i=r_{y_i}$, for $i=1,\ldots,k$, the announced result is satisfied.
\end{proof}

\begin{proof}[Proof of Proposition~\ref{Pr:existence}] 
Denote by $\mathcal{T}(z_0,z_1)$ the functional to maximise in the dual formulation \eqref{EqDualUOTProblem}.
 Since $x\mapsto \FF_i(x)$ is bounded below, $\FF^*_i(0)$ is finite, $i=0,1$. Hence, $\mathcal{T}(0,0)=-\FF^*_0(0)\rho_0(M)-
 \FF^*_1(0)\rho_1(M)$ is finite and  the supremum in \eqref{EqDualUOTProblem} is bounded below by $\mathcal{T}(0,0)$. Moreover, since $\FF^*_1$ is non decreasing, taking the local $c$-conjugate of $z_0\in C(M)$ improves the value of $\mathcal{T}$, i.e., $\mathcal{T}(z_0,\hat z_0)\geq \mathcal{T}(z_0,z_1) $ where 
  \begin{equation*}
\hat{z_0}(y) = \inf_{x\in\spt\rho_0} c(x,y) - z_0(x)\,.
\end{equation*}
Again, given a function $z_1$, taking the local $c$-conjugate of $z_1$ improves the value of $\mathcal{T}$, i.e.,
 $\mathcal{T}(\hat z_1,z_1)\geq \mathcal{T}(z_0,z_1)$,
 where
  \begin{equation*}
\hat{z_1}(x) = \inf_{y\in\spt\rho_1} c(x,y) - z_1(y)\,.
\end{equation*}
 Iterating this alternate optimization enables to restrict the optimization set to pairs of potentials that satisfy $z_1 = \hat{z}_0$ and $z_0 = \hat z_1$ (indeed, the local $c$-conjugate is an involution on its range).
 The value of $\mathcal{T}(z_0,z_1)$ does not depend on the behaviour of $z_i$ on $M\setminus\spt\rho_i$, therefore in the sequel we consider (any)  continuous extension of a function, defined only on the closed set $\spt\rho_i$, to the whole manifold. 
We  prove that the set 
\begin{equation*}
\mathcal{E}=\{(z_0,z_1)\in C(M)^2\mid \eqref{EqInequalityConstraint} \textrm{ is satisfied}, ~~\mathcal{T}(z_0,z_1)\geq\mathcal{T}(0,0) \text{ and }z_1 = \hat{z}_0,\,z_0 = \hat z_1\}
\end{equation*}
is equibounded and equi-Lipschitz, i.e., there exist  constants $A, B$ and $L>0$ such that for every pair $(z_0, z_1)\in\mathcal{E}$, and $i=0,1$, $B\leq z_i|_{\spt \rho_i}\leq A$, and $z_i|_{\spt \rho_i}$   is  $L$-Lipschitz continuous. 

Trivially, $\mathcal{E}$ is not empty, since it contains $(h,\hat h)$, $h(x)\equiv 0$. Let us start by equiboundedness of $\mathcal{E}$. We first demonstrate that there exist an upper bound for $\hat z_0$ on $\spt \rho_1$, that is uniform for every $(z_0,\hat z_0)\in\mathcal{E}$.
To this aim, consider $B(x_i,r_i)$ for $i = 1,\ldots,k$ given by Lemma \ref{ThLemma} for the measure $\rho_0$ such that 
\begin{equation*}
\sup_{y \in \spt\rho_1} \min_{i = 1,\ldots,k} c(x_i,y) \leq c_H(\rho_0,\rho_1) + 1\,.
\end{equation*}

Since $F_0^*(x) \geq \langle x,0 \rangle - F_0(0) = -F_0(0)$, for every $i \in 1,\ldots,k$, there holds
\begin{align*}
\mathcal{T}(0,0) \leq \mathcal{T}(z_0,\hat z_0) \leq  - \int_{B(x_i,r_i)}\FF_0^*(-z_0(x))\,d\rho_0(x) + \FF_0(0)\rho_0(M)   + \FF_1(0)\rho_1(M) . 
\end{align*}
Let $\tilde{z}_i = \sup\left\{z_0(x)\mid x \in B(x_i,r_i)\cap\spt\rho_0\right\}$. 
Then, for every $i$,
\begin{align}\label{equa}
\mathcal{T}(0,0)-\FF_0(0)\rho_0(M)   - \FF_1(0)\rho_1(M) \leq  - \FF_0^*(-\tilde z_i)\rho_0(B(x_i,r_i)).
\end{align}
Since $F_0^*$ is increasing, it is invertible and, for a given index $i \in\{ 1,\ldots,k\}$, there are two possibilities: either $\tilde z_i\geq -(F_0^*)^{-1}(0)$ or $\tilde z_i< -(F_0^*)^{-1}(0)$. 
Define
\begin{equation*}
    \kappa\eqdef\min\left\{-(F_0^*)^{-1}(0),\frac{\mathcal{T}(0,0)-\FF_0(0)\rho_0(M)   - \FF_1(0)\rho_1(M)}{\bar \delta} \right\},
\end{equation*}
where $\bar\delta=\min\{\rho_0(B(x_i,r_i), i=1,\dots k\}$.
In the first case,  $\tilde z_i\geq-(F_0^*)^{-1}(0)$, hence $\tilde z_i\geq\kappa$. In the second case, there holds $F_0^*(-\tilde z_i)\geq 0$ and equation \eqref{equa} implies
\begin{align*}
    \frac{\mathcal{T}(0,0)-\FF_0(0)\rho_0(M)   - \FF_1(0)\rho_1(M)}{\bar\delta} \leq  - \FF_0^*(-\tilde z_i).
\end{align*}
Since $\FF_0^*(x) \geq \langle x,1\rangle - F_0(1) = x$, we deduce that 
\begin{align*}
    \tilde z_i\geq \frac{\mathcal{T}(0,0)-\FF_0(0)\rho_0(M)   - \FF_1(0)\rho_1(M)}{\bar\delta}\geq\kappa.
\end{align*}
 In the end, $\tilde z_i$ is bounded below by the  constant $\kappa$ which depends only on $\mathcal{T}(0,0), \FF_0,\FF_1,\rho_0,\rho_1,\bar \delta$.

For every $y\in \spt\rho_1$, let $(x_{\bar i},r_{\bar i})$ given by Lemma \ref{ThLemma} such that $\sup_{x \in B(x_{\bar i},r_{\bar i})} c(x,y) \leq c_H(\rho_0,\rho_1) + 1$. 
Then
\begin{align*}
\hat z_0(y) &= \inf_{x\in \spt\rho_0} c(x,y) - z_0(x)\\
&\leq \inf_{x\in  B(x_{\bar i},r_{\bar i})\cap 
\spt\rho_0} c(x,y) - z_0(x)\leq c_H(\rho_0,\rho_1) + 1 + \inf_{B(x_{\bar i},r_{\bar i})\cap 
\spt\rho_0}- z_0(x)\,\\
&= 1+c_H(\rho_0,\rho_1) - \sup_{ B(x_{\bar i},r_{\bar i})\cap 
\spt\rho_0} z_0(x)= 1+c_H(\rho_0,\rho_1)-\tilde z_{\bar i}\\
&\leq  1+c_H(\rho_0,\rho_1)-\kappa.
\end{align*}
 Hence $\hat z_0$ is bounded above on $\spt \rho_1$ by $1+c_H(\rho_0,\rho_1)-\kappa$.
As a direct consequence,  $z_0$ is bounded below on $\spt\rho_0$ by $c_I-1-c_H(\rho_0,\rho_1)+\kappa$, where $c_I$ is the infimum of the cost function $c$.
By symmetry of the hypothesis on $\rho_0,\rho_1$, we obtain that there exists $A, B$, depending only on $\rho_0,\rho_1$, $\FF_0^*,\FF_1^*$ and $c_H(\rho_0,\rho_1), c_I$ such that $B\leq z_0|_{\spt\rho_0} \leq A$ and $B\leq \hat z_0|_{\spt\rho_1} \leq A$, for every $(z_0,\hat z_0)\in \mathcal {E}$.

We now prove that there exists a uniform constant $L$ such that for every pair $(z_0,z_1)\in\mathcal{E}$, $z_i|_{\spt\rho_i}$ is  Lipschitz continuous with constant $L$. Let $(z_0,z_1)\in\mathcal{E}$. By definition of $\mathcal{E}$, $z_1=\hat z_0$. Let $y\in \spt\rho_1$ and let $\bar x\in\spt\rho_0$ be such that 
\begin{equation}\label{equno}
z_1(y) = c(\bar x,y) - z_0(\bar x).
\end{equation}
Since $z_i$ is bounded by $A$ on $\spt\rho_i$, there holds  $c(\bar x,y)\leq 2 A$. Hence, by continuity, there exists $\bar \rho>0$ such that $\left\{t\in \spt\rho_1\mid d(y,t)<\bar \rho\right\}\subset c^{-1}(2A)$. Let $\tilde L>0$ be the Lipschitz constant of $c$ on $c^{-1}(2A)$. Then, for every $t\in\spt\rho_1$
\begin{equation}\label{eqdue}
\hat{z}_0(t) \leq c(\bar x, t) - z_0(\bar x)\,.
\end{equation}
If moreover, $d(y,t)\leq\bar\rho$, then subtracting \eqref{equno} from \eqref{eqdue} gives
\begin{equation*}
\hat{z}_0(y) - \hat{z}_0(t) \leq c(\bar x,y) - c(\bar x,t)\leq \tilde L d(y,t)\,.
\end{equation*}
By compactness, we infer that there exists a constant $L>0$, depending only on $\tilde L$ and not on $z_1$, such that $z_1|_{\spt\rho_1}$ is $L$-Lipschitz. 
By a symmetric argument, $z_0|_{\spt\rho_0}$ is $L$-Lipschitz.
Therefore $\mathcal{E}$ is not empty, equibounded and equi-Lipschitz. As a consequence, the existence of an optimal pair $(z_0,z_1)$ for \eqref{EqDualUOTProblem} with the required properties is obtained with a standard argument based on  Ascoli--Arzel\`a theorem for compactness and dominated convergence theorem for the convergence of the functional $\mathcal{T}$.
\end{proof}

As concerns uniqueness, an obvious sufficient condition is given by the following statement.  
\begin{proposition}\label{ThUniquenessofPotentials}
If $\FF_0^*$ and $\FF_1^*$ are strictly convex, the optimal pair $(z_0,z_1)$ is unique $\rho_0$ and $\rho_1$ a.e.
\end{proposition}

\begin{proof}
The maximization problem \eqref{EqDualUOTProblem} is strictly convex.
\end{proof}

Collecting the previous results leads to the existence and uniqueness of optimal Lipschitz potentials for \eqref{EqPrimalUOT} for some relevant costs in the literature. 
For instance, in usual applications outside mathematics, the Euclidean squared distance is often used. From the mathematical point of view, the case of 
\begin{equation*}
c(x,y)=-\log\left(\cos^2\left(d(x,y) \wedge \frac \pi 2\right)\right)
\end{equation*}
stands out since it appears in the static formulation of the Wasserstein-Fisher-Rao metric. Importantly, this cost is unbounded as well as its gradients, since it blows up when $d(x,y)$  is close to $\pi/2$.  
\begin{corollary}\label{Co:exilip}
Let $\FF_0(x) = \FF_1(x) = x\log(x) - x + 1$ and 
\begin{equation}\label{eq:costs}c(x,y) = \frac 12 d(x,y)^2, \textrm{ or }c(x,y)=-\log\left(\cos^2\left(d(x,y)\wedge {\delta \pi / 2}\right)\right)
\end{equation} for some $\delta >0$. Then, for every pair of admissible measures,  there exists a unique pair of Lipschitz continuous optimal potentials for the dual formulation \eqref{EqDualUOTProblem}.
\end{corollary}
Note that any pair of measures is admissible for the quadratic cost.

By Corollary~\ref{Co:exilip}, assumption~(3) in Theorem~\ref{ThReduction} is fulfilled and for the chosen entropy function, assumption~(1) is satisfied. Therefore, in the setting of Corollary~\ref{Co:exilip}, the regularity of unbalanced optimal transport is reduced to the regularity of the standard optimal transport and it can be inferred in different ways depending on the choice of the ambient space $M$. 
When $M = \R^d$, the quadratic cost supports regularity theorems for optimal transport. 
For the second cost in \eqref{eq:costs}, regularity results also hold for $M = S^d$ the unit sphere of dimension $d$ and for the sphere of radius $1/2$. Providing sufficient conditions of regularity of unbalanced optimal transport problems associated with costs in \eqref{eq:costs} is the object of the next section.

\subsection{Sufficient conditions for regularity of unbalanced optimal transport for two important costs.}\label{sec:leeli}

We focus on the two costs in \eqref{eq:costs}, which are of relevance in unbalanced optimal transport.
The first one is the most commonly used in practical applications, the Euclidean squared cost. The second one arises naturally from the dynamic formulation which was originally proposed to introduce this model.
In \cite{LMS}, the distances associated with those costs are named after  Gaussian-Hellinger for the quadratic case,  and Hellinger-Kantorovich for the other cost. The latter is also known as Wasserstein-Fisher-Rao distance (see for instance \cite{GeneralizedOT1,GeneralizedOT2}).

\smallskip

\noindent
\textbf{Gaussian-Hellinger distance: Euclidean space and spheres.}
Regularity in these two cases is an immediate consequence of Theorem~\ref{ThReduction}  and the regularity of optimal transport, for which sufficient conditions ensuring assumption~(2) in Theorem~\ref{ThReduction} are well-known. We simply detail the case of the Euclidean space, for which the following statement holds true, as a consequence of \cite[Theorem 3.3]{philippis2013mongeampre}. 
\begin{corollary}
Let $X, Y$ be convex sets in $\R^d$ and let $(\mu,\nu)\in\mathcal{M}_+(X)\times\mathcal{M}_+(Y)$ be a pair of measures which are absolutely continuous with respect to the Lebesgue measure, with densities $(f,g)$ bounded away from zero and infinity. 
Assume the entropy functions $\FF_0,\FF_1$ have strictly convex and differentiable Fenchel-Legendre transforms with infinite slope at $0$.

If $(f,g) \in C^{k,\alpha}(\overline{X})\times C^{k,\alpha}(\overline{Y})$ for some positive integer $k$ and $\alpha \in (0,1)$,  then, the pair of optimal potentials $(z_0,z_1)$ in the dual formulation \eqref{EqDualUOTProblem} for the quadratic cost $\frac 12 \| x - y\|^2$ belongs to  $C^{k+2,\alpha}(X)\times C^{k+2,\alpha}(Y)$ and $\nabla z_0$ is a $C^{k+1,\alpha}$-diffeomorphism between $\overline{X}$ and $\overline{Y}$.
\end{corollary}

\smallskip

\noindent \textbf{Wasserstein-Fisher-Rao distance.}
We consider the case of a $d$-dimensional Riemannian manifold $M$ having constant sectional curvature, i.e., $M$ may be the Euclidean space, a  $d$-sphere, or the hyperbolic space and 
\begin{equation}\label{costo}c(x,y)=-\log\left(\cos^2\left(d(x,y)\wedge \frac{\pi}{2}\right)\right)\,.
\end{equation} 
The purpose of this section is to provide sufficient conditions to ensure assumption~(2) in Theorem~\ref{ThReduction} based on the study of the so-called Ma-Trudinger-Wang tensor for the cost \eqref{costo} on such manifolds. 
Indeed, the relation between such tensor and smoothness of optimal transport maps has been completely understood in the work \cite{MTWregularity}. We also refer the reader to \cite[Chapter 12]{villanioldnew} and references therein for a comprehensive treatment of the subject.
Roughly speaking, besides conditions on densities and supports of source and target measures, one is led to check whether the MTW tensor is nonnegative.  In particular: MTW weak condition states that MTW tensor must be nonnegative for every pair of points and every pair of $c$-orthogonal vectors; MTW strong condition states that MTW weak condition holds true and the tensor vanishes only at vanishing vectors. 
The following result provides instances of Riemannian manifolds where MTW strong or weak conditions hold.

\begin{proposition}\label{contazzi}
Let $M$ be a Riemannian manifold with constant sectional curvature and let $c:M \times M\to \R\cup\{+\infty\}$ be as in \eqref{costo}.

Then 
\begin{itemize}
    \item[(i)] MTW weak condition for $c$ fails if $M$ is either the Euclidean space $\R^d$, either the hyperbolic space $\mathbb{H}^d$ or the $d$-sphere of radius $R> 1$ with the induced metric;
    \item[(ii)] MTW weak condition holds for $c$ if $M$ is the $d$-sphere of radius $1$ with the induced metric;
    \item[(iii)] MTW strong condition holds for $c$ if $M$ is the   $d$-sphere of radius $R=1/2$    with the induced metric.
\end{itemize}
\end{proposition}
As a consequence, combining Proposition~\ref{contazzi} with Theorem~\ref{ThReduction}, unbalanced optimal transport on spheres or radius $1$ and $1/2$ features smoothness.

MTW tensor for costs of the type $c(x,y)=l(d(x,y))$ was analyzed in \cite{LeeLi} for even smooth functions $l: \R\to [0,+\infty)$ having invertible derivative. In particular,  authors characterize MTW weak and strong conditions on manifolds with constant sectional curvature in terms of some computable explicit functions, see~\cite[Theorem~5.3]{LeeLi}.

\begin{proof} We start by recalling the main results in \cite{LeeLi}. Consider a cost function $J(x,y)=l(d(x,y))$, where $l:\R\to [0,+\infty[\to\R$ is a smooth, even function such that $l''(s)>0$. Set $h(s)=(l')^{-1}(s)$. Then the $J$-exponential map can be computed as  
\begin{equation*}
J\textrm{-}\exp_x(v)=\exp_x\left(\frac{h(|v|)}{|v|}v\right),
\end{equation*}
where $\exp_x$ denotes the Riemannian exponential on $M$ and $|v|=\sqrt{g_x(v,v)}$ denotes the norm with respect to the metric tensor on $M$ (Recall that the $J$-exponential map is defined by the identity $J\textrm{-}\exp_x(v)=y$ if and only if $v=-\nabla_xJ(x,y)$). By definition, the $MTW$ tensor is
\begin{equation*}
MTW_x(u,v,w)=-\frac 3 2 \partial_s^2\partial_t^2|_{s=t=0}J(\exp_x(t u),J\textrm{-}\exp_x(v+sw)),
\end{equation*}
where $x\in M$, and $u,v,w$ are tangent vectors at $x$.
Define $A(s)=\frac{1}{h(s)}$, and
\begin{equation*}
B(s)=\begin{cases}
s \coth(h(s)),& \textrm{if } M=\R^d,\\
\frac{s}{h(s)},& \textrm{if } M=\mathbb{H}^d,\\
s \cot(h(s)),& \textrm{if $M$ is the unit sphere}.\end{cases} 
\end{equation*}
By \cite[Proposition~5.1]{LeeLi}, whenever $u$ and $w$ are $J$-orthogonal, the MTW tensor  can be simplified to
\begin{equation*}
MTW_x(u,v,w)=-\frac 3 2 \left(\alpha(|v|)|u_0|^2|w_0|^2+\beta(|v|)|u_0|^2|w_1|^2+\gamma(|v|)|u_1|^2|w_0|^2+\delta(|v|)|u_1|^2|w_1|^2\right),
\end{equation*}
where $u=u_0+u_1$, $w=w_0+w_1$,   $u_0,w_0\in\on{span}\{v\},u_1,w_1\in(\on{span}\{v\})^\perp$ and coefficients are given by
\begin{eqnarray}\label{eqal}
\alpha(s)&=&\frac{s^2A''(s)+6(A(s)-B(s))-4 s (A'(s)-B'(s))}{s^2},\\\label{eqbe}
\beta(s)&=&\frac{s A'(s)-2(A(s)-B(s))}{s^2},\\\label{eqga}
\gamma(s)&=&B''(s),\\\label{eqde}
\delta(s)&=&\frac{B'(s)}{s},
\end{eqnarray}
in terms of functions $A, B$ defined above. By Theorem~5.3 in \cite{LeeLi}, the MTW tensor satisfies MTW weak condition if and only if, for every $s\in[0,|l'(D)|]$, with $D$ the diameter of $M$, four inequalities hold
\begin{equation}\label{condleeli}
    \beta(s)\leq 0, ~~\gamma(s)\leq 0, ~~\delta(s)\leq 0,~~\alpha(s)+\delta(s)\leq 2\sqrt{\beta(s)\gamma(s)}. 
\end{equation}
Moreover, MTW strong condition holds if and only if the four inequalities are strict for every $s\in (0,|l'(D)|]$.

Note that cost $c$ in \eqref{costo} is of the type $l(d(x,y))$, for $l(s)=-\log(\cos^2(s))$. We compute explicitly functions $A,B$ for the hyperbolic space and for the Euclidean space. In both cases, $\beta(0)>0$, whence MTW weak condition fails.

When $M$ is the $d$-sphere of radius $R\in(0,+\infty)$, we interpret the cost $c$ in \eqref{costo} as $c(x,y)=l_R(d(x,y))$ where $l_R(x,y)=-\log(\cos^2(Rs))$. Hence we set   $B(s)=s \cot(h_R(s))$, with $h_R=(l_R')^{-1}$ and apply \cite[Proposition~5.1]{LeeLi} to compute the MTW tensor on the $d$-sphere of radius $R$ by means of the MTW tensor on the unit $d$-sphere with rescaled distance. Note that MTW conditions (weak or strong) must hold for 
$s\in [0, |l^{'}_R(D)|]$, where $D=\pi$ is the diameter of the unit sphere.

Computing explicitly, 
$\alpha(0) = \beta(0)=\gamma(0)=\delta(0)=\frac13\left(1-\frac{1}{R^2}\right)$. Therefore we conclude that when $R>1$ MTW weak condition fails. On the other hand, an explicit computation gives
\begin{eqnarray*}
\textrm{for } R=1,&&\alpha(s)=\beta(s)=\gamma(s)=\delta(s)\equiv 0,\\
\textrm{for } R=\frac12,&&\alpha(s)=\beta(s)=\gamma(s)=\delta(s)\equiv -1\,.\\
\end{eqnarray*}
Hence for $R=1$ MTW weak condition holds and MTW vanishes on $c$-orthogonal vectors, whereas for $R=1/2$ MTW strong condition holds.
\end{proof}
\begin{remark}
We refer the reader to Section~\ref{SecCconvex} for an alternative way of studying the MTW tensor associated with the cost \eqref{costo} on $M$. In that section, we study the link between the aforementioned tensor and the  MTW tensor associated with the quadratic cost on the cone $\cone(M)$ (see Section~\ref{secsemicouplings} for the definition of the cone and its Riemannian structure).  
\end{remark}
A natural question is whether MTW condition (weak or strong) holds true for spheres of radius $R\in (0,1/2)\cup(1/2,1)$. To this aim, it is sufficient to show that the four functions $\beta_R(\cdot),\gamma_R(\cdot), \delta_R(\cdot), (\alpha_R+\delta_R-2 \sqrt{\beta_R\gamma_R})(\cdot)$ are negative, respectively strictly negative, for every $s\in (0,|2 R \tan(\pi R)|]$, for MTW weak condition, respectively MTW strong condition, to hold. We end this section by  proving that, 
for $R\in(0,1/2)$, $\beta_R(s)\leq 0$ and $\gamma_R(s)\leq 0$, for every $s\in(0, 2 R \tan(\pi R))$. As for the sign of the other two functions, we can only hint at their negativity by analyzing the shape of their 0-level sets. 
  
 Using \eqref{eqal}, \eqref{eqbe}, \eqref{eqga}, \eqref{eqde}, an easy computation gives
\begin{eqnarray*}
\alpha_R(s)&=&\frac{12R^2}{s^2}-\frac2s\cot\left(\frac{1}{R}\arctan(s/(2 R))\right)-\frac{8}{s^2+4 R^2}\csc^2\left(\frac1 R\arctan(s/(2 R))\right),\\
\beta_R(s)&=&\frac{2}{s^2}\xi(s)%\left(s \cot\left(\frac{1}{R}\arctan(s/(2 R))\right)-2 R^2\right)
,\\
\gamma_R(s)&=&\frac{8 \csc^2\left(\frac1 R\arctan(s/(2 R))\right)}{(s^2+4 R^2)^2}\xi(s)%\left(s \cot\left(\frac{1}{R}\arctan(s/(2 R))\right)-2 R^2\right)
,\\
\delta_R(s)&=&\frac1s\cot\left(\frac{1}{R}\arctan(s/(2 R))\right)-\frac{2}{s^2+4 R^2}\csc^2\left(\frac1 R\arctan(s/(2 R))\right),\\
\end{eqnarray*}
 where the auxiliary function $s\mapsto\xi(s)$ is defined by
 \begin{equation*}
 \xi(s)=s \cot\left(\frac{1}{R}\arctan(s/(2 R))\right)-2 R^2.
\end{equation*}
We are going to show that, for every $R\in(0,1/2)$ and every $s\in(0, 2 R \tan(\pi R))$, $\xi(s)<0$. 
Note that for $R\in(0,1/2)$, $\frac{\arctan(s/(2 R))}{R}\in(0, \pi)$. Hence $\xi(s)<0$ is equivalent to 
$s \cot(\frac{\arctan(s/(2 R))}{R})< 2 R^2$ which in turn is equivalent to 
$\frac{\arctan(s/(2 R))}{R}> \mathrm{arccot}(\frac{2 R^2}{s})$. 
Using $\arctan x = \mathrm{arccot}(1/x)$, the last inequality is equivalent to 
$\mathrm{arccot}(2 R/s)> R \mathrm{arccot}(\frac{2 R^2}{s}).
$
Set $v=2 R/s$ and define
$k(v)= \mathrm{arccot}(v)-R \mathrm{arccot}(R v)$. To show that $\xi(s)<0$ it is sufficient to prove that $k(v)>0$ on $(0, +\infty)$. This is an easy consequence of the fact that $k(0)=\pi/2(1-R)>0$, 
$\lim_{v\to+\infty} k(v)=0$ and 
\begin{equation*}
k'(v) =\frac{R^2}{1 + R^2 v^2}-\frac{1}{1 + v^2} < 0 , v\in(0,+\infty).
\end{equation*}
%%%%%%IN CASE OF FIGURES
%To test the last two conditions in \eqref{condleeli}, let us
 %plot  the $0$-level sets of the functions $\delta_R(\cdot), (\alpha_R+\delta_R-2 \sqrt{\beta_R\gamma_R})(\cdot)$ in the region $(R,s)\in (0,1)\times(0,25)$. We plot also the function $w(R)=|l^{'}_R(\pi)|=|2 R \tan(\pi R)|$. Recall that the MTW strong condition holds if both functions  $ \delta_R(\cdot), (\alpha_R+\delta_R-2 \sqrt{\beta_R\gamma_R})(\cdot)$ are strictly negative for every $s\in (0,|2 R \tan(\pi R)|]$. The following figure portraits the zero-level set of $\delta_R(\cdot)$, respectively $(\alpha_R+\delta_R-2\sqrt{\beta_R\gamma_R})(\cdot)$, together with the dashed curve $s=|2 R \tan(\pi R)|$ in the $(s,R)$ coordinate plane for $s\in[0,40]$ and $R\in[0,1]$. Since the set of points $\{(s,R)\mid \delta_R(s)=0$ lies above $s=|2 R \tan(\pi R)|$ and $\delta_{0.4}(2)<0$ we infer that $\delta_R(s)<0$ for every $(s,R)$ such that $(R\in(0,1/2)$ and $s<2 R \tan(\pi R)$. Similarly, from figure... we infer that  $\alpha_R(s)+\delta_R(s)-2\sqrt{\beta_R(s)\gamma_R(s)}<0$ for every $(s,R)$ such that $(R\in(0,1/2)$ and $s<2 R \tan(\pi R)$.
%  
 %%%%%
 
%\begin{remark}
 %For $r = 1$ the MTW tensor vanishes identically on $c$-orthogonal vectors.
%For  $r = 1/2$  the MTW tensor is given by the formula
%\begin{equation}
%MTW_{(0,\varphi), (\theta,\varphi)}(\xi,\tilde\xi)=\frac32\|\xi\|^2,
%\end{equation}
%where $\xi=(a_1,a_2)$, $\tilde\xi=(\nabla_{x}c)^{-1}|_{(0,\varphi), %(\theta,\varphi)}(-a_2,a_1)$.
%\end{remark}

\section{The Wasserstein-Fisher-Rao metric}\label{SecStaDyn}

In this section, we detail the case of the Wasserstein-Fisher-Rao (WFR) metric on a smooth compact Riemannian manifold $M$, which is the cornerstone of unbalanced optimal transport as introduced in \cite{new2015kondratyev,GeneralizedOT1,LMS}.
Recall that the Wassertein-Fisher-Rao   corresponds to the cost function given in \eqref{costo}  
and to the Kullback-Leibler divergence for the marginal penalization (i.e., both entropy functions are given by $\FF(x) = x\log(x) - x + 1$.
First, we prove the equivalence of several definitions of this metric.
In particular, we introduce an equivalent of the Monge formulation of standard OT to this unbalanced setting.
Using this formulation we prove the existence of unbalanced optimal transport maps and an unbalanced version of Brenier polar factorization theorem on the automorphism group of the cone $\cone(M)$ see Theorem \ref{Th:polardecompo}.
A regularity theory for such maps is obtained in section~\ref{SecDualRegularity}: we provide here  its link to an unbalanced Monge-Ampère equation, see section~\ref{secMonge}.

\subsection{Equivalent formulations of  WFR metric.}

As in classical optimal transport,  the Wasserstein-Fisher-Rao metric can be defined in many ways.
Here we detail five of them, namely: dynamical, semi-couplings, dual, Kantorovich and Monge  formulations.
The Kantorovich formulation is the one introduced in Definition~\ref{Def:UOTkanto} and the dual formulation is given in Proposition~\ref{prop:formulationduale}. 
For the sake of clarity, we instantiate them hereafter.  We chose to deal with Monge formulation in another section, namely Section~\ref{sec:mongeformulation}, since it is deduced from a geometric construction on the cone. 

The starting point of all these formulations is 
the dynamical formulation of the WFR metric which appears naturally as a generalization of the Benamou-Brenier formula by introducing a source term in the continuity equation. This is the formulation we first present below.
 
In the sequel, let $(M,g)$ be a compact Riemannian manifold, let $\on{vol}$ denote the Riemannian volume on $M$ and let $\Div$ denote the divergence of a vector field with respect to $\on{vol}$. 
\subsubsection{Dynamical formulation of WFR metric}

Given $\rho_0,\rho_1\in\mathcal{M}_+(M)$ and $a,b > 0$, we start by the following optimization problem 

\begin{equation*}%\label{EqWFR}
\inf_{\rho,v, \alpha} \frac12 \int_0^1  \int_M  \left(a^2g_x(v(x),v(x)) + b^2 \alpha^2(t,x) \right)\d \rho_t(x)  \d t 
\end{equation*}
under the constraints of the generalized continuity equation, with time boundary conditions 
\begin{equation*}
\partial_t \rho + \Div(\rho v) = \alpha \rho\,, \rho(0,\cdot)=\rho_0,\rho(1,\cdot)=\rho_1\,.
\end{equation*}
Here the control variables are $\alpha$, the growth rate (also called {\it Malthusian parameter}) and $v$, a vector field, both depending on time $t$ and position $x\in M$. 
\begin{remark}
For $\alpha\equiv 0$, the dynamic formulation above is the well-known  Benamou-Brenier formulation of the optimal transport problem \cite{benamou2000computational}.
\end{remark}
We now give the definition, relying on convexity, which allows us to account for every positive Radon measure and not only those with density with respect to the reference volume measure.
\begin{definition}[Dynamical formulation of $\on{WFR}$ metric]\label{defwfr} Let $\rho_0,\rho_1 \in \mathcal{M}_+(M)$, the {\it WFR metric} is defined by
\begin{equation*} %\label{Extension1}
\on{WFR}^2(\rho_0,\rho_1) \eqdef \inf_{\rho,\m,\mu} \mathcal{J}(\rho,\m,\mu)  \, ,
\end{equation*}
where 
\begin{equation}\label{EqWFRonM}
\mathcal{J}(\rho,\m,\mu) =  a^2 \iint_{[0,1]\times M} \frac {g^{-1}_x(\tilde{\m}(t,x),  \tilde{\m}(t,x))}{\tilde{\rho}(t,x)} \,\ud \nu(t,x) \, +  b^2\iint_{[0,1]\times M} \frac {\tilde{\mu}(t,x)^2}{\tilde{\rho}(t,x)} \,\ud \nu(t,x)
\end{equation} 

over the set $(\rho,\m,\mu)$ satisfying $\rho \in \mathcal{M}_+([0,1] \times M)$, $\m \in (\Gamma_M^0([0,1] \times M,TM))^*$ which denotes the dual of time-dependent continuous vector fields on $M$ (time-dependent sections of the tangent bundle), $\mu \in \mathcal{M}_+([0,1] \times M)$
subject to the constraint
\begin{equation}\label{Eq:ContinuityEquation}
\iint_{[0,1]\times M} \partial_t f\ud \rho +\iint_{[0,1]\times M} (\m(\nabla_x f)  - f\ud\mu)  =  \int_M f(1,\cdot) \ud \rho_1 -  \int_M f(0,\cdot) \ud \rho_0
\end{equation}
satisfied for every test function $f \in C^1([0,1]\times M,\R)$. Moreover, $\nu \in \mathcal{M}_+([0,1]\times M)$ 
is chosen such that $\rho,\m,\mu$ are absolutely continuous with respect to $\nu$ and $\tilde{\rho},\tilde{\m},\tilde{\mu}$ denote their Radon-Nikodym derivative with respect to $\nu $. 
\end{definition}
Note that due to the one-homogeneity of the formulas with respect to $(\tilde{\rho},\tilde{\m},\tilde{\mu})$, the functional $\mathcal{J}$ is well-defined, i.e., it does not depend on the choice of the dominating measure $\nu$. Moreover, the divergence is defined by duality on the space $C^1(M)$.
Formula \eqref{EqWFRonM} in Definition~\ref{defwfr} is called dynamic since the time variable is involved and only length-space structures can be defined in this way. It is of interest to show that the variational problem admits a so-called static formulation that does not involve the time variable. 

\subsubsection{Semi-couplings formulation}\label{secsemicouplings}
The semi-couplings formulation already appears in \cite{GeneralizedOT2} and in another form in \cite{LMS}. In both references, equivalence between semi-couplings and dynamical formulation is proved in the Euclidean case.
We now extend those results to a Riemannian setting.

Given $\rho_0,\rho_1\in\mathcal{M}_+(M)$, set
\begin{align*}
		\Gamma(\rho_0,\rho_1) \eqdef \left\{
			(\eta_0,\eta_1) \in \big(\mathcal{M}_+(M^2)\big)^2 \colon
			p^1_* \eta_0 = \rho_0,\, p^2_* \eta_1= \rho_1
			\right\}\,,
\end{align*}
where $p^1$ and $p^2$ denote the projection on the first and second factors of the product $M^2$.
Moreover, consider the cone 
 \begin{equation*}
 \cone(M)=\{(x,r)\mid x\in M, r>0\},
 \end{equation*} endowed with the Riemannian metric
\begin{equation}\label{eqmetricacono}
h_{(x,r)}=a^2r^2g_x+4 b^2\ud r^2,
 \end{equation} 
where $g$ is the Riemannian metric on $M$, and $a,b$ appear in the definition of WFR metric. Finally, denote by $\dcone$ the distance on $\cone(M)$ associated with the Riemannian metric $h$.
 
\begin{theorem}[Semi-couplings formulation of $\on{WFR}$ metric]\label{Th:DynamicToStaticOnRiemannianManifolds}
The WFR distance satisfies
\begin{align}
	\label{Eq:StaticProblem}
		\on{WFR}^2(\rho_0,\rho_1)  = \min_{(\eta_0,\eta_1) \in \Gamma(\rho_0,\rho_1)}
			\int_{M^2} \ddcone \left(\left(x,\sqrt{\frac{\ud \eta_0}{\ud \eta}}\right),\left(y,\sqrt{\frac{\ud \eta_1}{ \ud \eta}}\right)\right) \ud \eta(x,y) \,,
	\end{align}
where  $\eta$ is any measure that dominates $\eta_0,\eta_1$.
\end{theorem}
The functional 
\begin{equation*}
\mathcal{S}(\eta_0,\eta_1) \eqdef \int_{M^2} \ddcone \left(\left(x,\sqrt{\frac{\ud \eta_0}{ \ud \eta}}\right),\left(y,\sqrt{\frac{\ud \eta_1}{ \ud\eta}}\right)\right) \ud \eta(x,y)
\end{equation*} is well-defined, i.e., it does not depend on the choice of the measure $\eta$. Indeed, the squared distance function $\ddcone$ is two-homogeneous with respect to dilation of the mass variables, since $h_{(x,\lambda r)}=a^2(\lambda r)^2g_x+4 b^2\lambda^2\ud r^2$. As a consequence of Rockafellar's theorem \cite[Theorem 5]{rockafellar1971integrals}, $\mathcal{S}$ is convex and lower semicontinuous on the space of Radon measures as the Legendre-Fenchel transform of a convex functional on the space of continuous functions.

Our proof of Theorem \ref{Th:DynamicToStaticOnRiemannianManifolds} is an adaptation to the Riemannian case of the one in  \cite[Theorem~4.3]{GeneralizedOT2}, to which we refer the reader for technical details. The same reasoning, based on a simple regularization argument that is intrinsic on Riemannian manifolds, applies under minor adaptations to the standard Wasserstein distance $W_2$  on Riemannian manifolds, see for instance the comments in \cite[Remark 8.3]{VillaniTOT}. A different proof of the equivalence between dynamical and semi-couplings formulation for the Wasserstein distance $W_2$ in the Riemannian setting is given in \cite{UserGuideOT} which uses the Nash isometric embedding theorem.
%Below is the proof of Theorem \ref{Th:DynamicToStaticOnRiemannianManifolds} which is an adaptation to the Riemannian case of the proof in \cite{GeneralizedOT2}. In particular, not all the details of the proof are given since they can be found in \cite{GeneralizedOT2}. Note also that this proof, under minor adaptations, applies to the standard Wasserstein $L^2$ metric on Riemannian manifolds, see for instance the comments in \cite[Remarks 8.3]{VillaniTOT}. A proof of the standard Wasserstein case is given in \cite{UserGuideOT} which uses the Nash isometric embedding theorem. The proof below does not use it and develop a simple regularization argument which is intrinsic on Riemannian manifolds.
\begin{proof}[Proof of Theorem \ref{Th:DynamicToStaticOnRiemannianManifolds}]
First of all, the set $\Gamma$ is weak$^*$ closed, and the functional $\mathcal{S}$ is weakly continuous and lower semicontinuous. Therefore, the fact that the minimum for $\mathcal{S}$ is attained follows by application of the direct method of calculus of variations. 

Since $\dcone$ is a distance on $\cone(M)$, one can prove that $\mathcal{S}$ is a distance on $\mathcal{M}_+(M)$ and $\mathcal{S}$ is continuous w.r.t. the weak$^*$ topology, as done in \cite{GeneralizedOT2}. 

We claim that for every pair of measures $(\rho_0,\rho_1)$ that are a finite linear combinations of Dirac masses, the inequality
\begin{equation*}
\mathcal{S}(\rho_0,\rho_1) \geq \on{WFR}^2(\rho_0,\rho_1), 
\end{equation*}
 holds.
 To see this, note that for $\rho_0 = \sum_{i} a_i \delta_{x_i}$ and $\rho_1 = \sum_{j} b_j \delta_{y_j}$, for finite sets of points $\{x_i,y_j\}_{i,j}\subset M$, 
 the minimization problem \eqref{Eq:StaticProblem} can be reduced to a linear optimization problem in finite dimension. Indeed, the optimal semi-couplings can be proved to have support on the product of the supports of $\rho_0$ and $\rho_1$. As a consequence, the optimal semi-couplings can be written as $\gamma^k = \sum_{i,j} m^k_{i,j} \delta_{(x_i,y_j)}$ for $k = 0,1$. Then, one has
\begin{align*}
\mathcal{S}(\rho_0,\rho_1) &= \sum_{i,j} \ddcone\left((x_i,m_{i,j}^0),(y_j,m_{i,j}^1)\right) \\
 & \geq  \sum_{i,j} \on{WFR}^2(m_{i,j}^0\delta_{x_i},m_{i,j}^1\delta_{y_j}) \geq \on{WFR}^2(\rho_0,\rho_1)\,,
\end{align*}
where the first inequality comes from the fact that the distance on the cone (with mass coordinates) for a geodesic $(x(t),m(t))$ is given by the evaluation of $\on{WFR}$ on the path $m(t)\delta_{x(t)}$. The second inequality is given by subadditivity of $\on{WFR}^2$. 
By density of finite linear combination of Dirac masses and weak$^*$ continuity of both $\on{WFR}$ and $\mathcal{S}$, the inequality  $\mathcal{S}(\rho_0,\rho_1) \geq \on{WFR}^2(\rho_0,\rho_1)$ holds on $(\mathcal{M}_+(M))^2$.
\par
We now prove the reverse inequality which follows using the convexity of $(\rho_0,\rho_1)\mapsto \on{WFR}^2(\rho_0,\rho_1)$. By subadditivity of $\on{WFR}^2$, one has, for every $\rho_2\in\mathcal{M}_+(M)$
\begin{equation*}
\on{WFR}^2(\rho_0 + \rho_2,\rho_1 + \rho_2) \leq \on{WFR}^2(\rho_0,\rho_1)\,.
\end{equation*}
Using the triangular inequality and the fact that the $\on{WFR}$ metric is bounded above (up to a multiplicative constant) by the Hellinger distance, we also have, for $\varepsilon_1>0$
\begin{equation*}
\on{WFR}(\rho_0,\rho_1) \leq \on{WFR}(\rho_0+\varepsilon_1 \on{vol},\rho_1+\varepsilon_1 \on{vol}) + 2\,\text{cst}\,\sqrt{\varepsilon_1} \,.
\end{equation*}
Let us be more precise on the previous inequality.
Consider a triplet $(\rho,\m,\mu)$ which is a solution to the continuity equation \eqref{Eq:ContinuityEquation}, then so does the triplet $(\rho + \varepsilon_1 \on{vol},\m,\mu)$ satisfying the boundary conditions $\rho(0,\cdot) = \rho_0$, $\rho(1,\cdot) = \rho_1$. Note that $\varepsilon_1 \on{vol}$ is constant in time and space. In addition, it is obvious that
$$\mathcal{J}(\rho+ \varepsilon_1 \on{vol},\m,\mu) \leq \mathcal{J}(\rho,\m,\mu)\,.$$
To prove the final result, it suffices to prove that $\mathcal{S}(\rho_0+\varepsilon_1 \on{vol},\rho_1+\varepsilon_1 \on{vol}) \leq \mathcal{J}(\rho+ \varepsilon_1 \on{vol},\m,\mu) + \varepsilon_0$ for any $\varepsilon_0 >0$. This will be done via a smoothing argument which is standard in the Euclidean case using convolution but has never been adapted, to the best of our knowledge, to work on Riemannian manifolds (see \cite[Remarks 8.3]{VillaniTOT}).
\par
Our goal is to prove that there exists a path of smooth quantities $(\rho_\varepsilon,\m_\varepsilon,\mu_\varepsilon)$ for which $\mathcal{J}(\rho_\varepsilon,\m_\varepsilon,\mu_\varepsilon)$ is close to $\mathcal{J}(\rho,\m,\mu)$ and $\rho_\varepsilon$ is strictly positive and the time endpoints of the path are close in the weak$^*$ topology. 
The conclusion can then be obtained by integrating the flow defined by the vector field $(\m_\varepsilon/\rho_\varepsilon,\mu_\varepsilon/\rho_\varepsilon)$. It gives that 
$\mathcal{S}(\rho_\varepsilon(0),\rho_\varepsilon(1)) \leq \mathcal{J}(\rho_\varepsilon,\m_\varepsilon,\mu_\varepsilon)$ and the conclusion is similar to the Euclidean case \cite[Theorem 5]{GeneralizedOT2}.
\par
By compactness of $M$, it is sufficient to locally smooth the path on $M$ by iteration of this smoothing. Therefore, we will work on a chart $U$ around a point $x_0 \in M$. By Moser's lemma, it is possible to choose the chart such that the volume form is the Lebesgue measure. \par
\textbf{Averaging over perturbations of identity.}
We construct perturbations (of compact support) of the identity which will be local translations around $x_0$ and which will play the role of the translations in the standard convolution formula. We consider a ball $B(x_0,r_0)$ and a function $u$ whose support is contained in $B(x_0,r_0)$ and is constant equal to $1$ on $B(x_0,r_1)$ for $0<r_1 < r_0$.
For a given vector $v \in \R^d$, we consider the map $\Phi_v(x) = x + u(x)v$ which is a smooth diffeomorphism. We extend $\Phi$ to the whole manifold $M$ by defining it as identity outside of $U$.
\par
Let $k: \R^{d+1} \to \R_+$ be a smooth symmetric and nonnegative function whose support is contained in the unit ball and such that $\int k(y) \ud y = 1$ and define for $\varepsilon>0$, $k_\varepsilon(x) = k(x/\varepsilon)/\varepsilon^{d+1}$ whose support is thus contained in the ball of radius $\varepsilon$. We define the mollifier $\mathsf{k}_\varepsilon \, \star $ acting on $f \in C([0,1] \times U,\R)$ by
\begin{equation*}
(\mathsf{k}_\varepsilon \star f) (s,x) = \int_{\R}\int_{U} k_\varepsilon(s,v)f(t+s,\Phi_v^{-1}(x))  \ud v  \ud s\,,
\end{equation*}
which is well defined for $\varepsilon$ small enough, extending the function outside the time interval $[0,1]$ as a constant. Moreover, for $\varepsilon$ sufficiently small, it coincides with the usual convolution on a  neighborhood of $x_0$. By duality, it is well-defined on Radon measures and extends trivially to vector-valued measures as follows
\begin{align*}
(\mathsf{k}_\varepsilon \star \rho) (s,x) = \int_{\R}\int_{U} k_\varepsilon(s,v)(\Phi_v)_*(\rho(t+s))  \ud v  \ud s\,, \\
(\mathsf{k}_\varepsilon \star m) (s,x) = \int_{\R}\int_{U} k_\varepsilon(s,v)\on{Ad}_{\Phi_v^{-1}}^*(\m(t+s))  \ud v  \ud s\,.
\end{align*}

We consider the path $(\Phi_v)_*(\rho)$ which satisfies the continuity equation for the triple of measures $\left((\Phi_v)_*(\rho),\on{Ad}_{\Phi_v^{-1}}^*(\m),(\Phi_v)_*(\mu)\right)$ and average over $v$ to consider
\begin{equation*}
(\rho_\varepsilon,\m_\varepsilon,\mu_\varepsilon) = \left(\mathsf{k}_\varepsilon \star \rho,\mathsf{k}_\varepsilon \star \m,\mathsf{k}_\varepsilon \star \mu \right)\,.
\end{equation*}
As a convex combination, this path satisfies the continuity equation and the boundary conditions are close in the weak$^*$ topology when $\varepsilon$ tends to $0$. An important remark is that, for $\varepsilon$ small enough, $\mathsf{k}_\varepsilon \star \on{Ad}_{\Phi_v^{-1}}^*(\m)$ reduces to the standard convolution on $m$ in a small neighborhood of $x_0$ since $D\Phi_v = \on{Id}$ in a neighborhood of $x_0$ since $u \equiv 1$ on $B(x_0,r_1)$.
\par
\textbf{Use of convexity of $\mathcal{J}$:} 
For notation convenience, we denote by $f$ the integrand of $\mathcal{J}$ and we make the abuse of notation to use $\rho,\m,\mu$ instead of their corresponding densities w.r.t. $\nu$ a dominating measure. 
\par The change of variables $y  = \Phi_v^{-1}(x)$ (we use one-homogeneity hereafter) leads to 
\begin{multline*}
\mathcal{J}(\rho_\varepsilon,\m_\varepsilon,\mu_\varepsilon) =  \int_{[0,1] \times M} f\left(x,(\rho_\varepsilon,\m_\varepsilon,\mu_\varepsilon) \right) \ud \nu(x) \leq \\ \int_{\R}\int_{U} \int_{[0,1]\times M} k_\varepsilon(s,v)f(\Phi_v(y),(\rho(t+s),D\Phi_v(t,y)\m(t+s),\mu(t+s))) \ud \nu(t,y) \ud t \ud s  \ud v\,.
\end{multline*}
Moreover, since the metric $g$ on $M$ is smooth and in particular uniformly continuous on $M$ and since $\| D\Phi_v - \on{Id} \| \leq \text{cst} \| v \| $ for a constant that only depends on $u$, we have, for any $\varepsilon_2>0$, the existence of $\delta >0$ such that if $\|v\| \leq \delta$ then,
\begin{equation*}
|g_x(w,w) - g_{\Phi_v(x)}(D\Phi_v(x)w,D\Phi_v(x)w) | \leq \varepsilon_2\, g_x(w,w)\,,
\end{equation*}
 for every $w \in T_xM$.
Therefore, a direct estimation leads to 
\begin{multline*}
\left| \int_{\R \times M}k_\varepsilon(s,v)f(\Phi_v(x),(\rho(t+s),\m(t+s),\mu(t+s))) \ud \nu(t,x)  -  \int_{[0,1]\times M} f(x,(\rho(t),\m(t),\mu(t))) \ud \nu(t,x) \right| \\ \leq \varepsilon_2  \mathcal{J}(\rho,\m,\mu)\,,
\end{multline*}
and as a consequence the desired result,
\begin{equation*}
 \mathcal{J}(\rho_\varepsilon,\m_\varepsilon,\mu_\varepsilon)  \leq \mathcal{J}(\rho,\m,\mu)  +\varepsilon_2  \mathcal{J}(\rho,\m,\mu)\,.
\end{equation*}
Since this averaging reduces to standard convolution in the coordinate chart $U$ in a small neighborhood of $x_0$, it implies that 
$(\rho_\varepsilon,\m_\varepsilon,\mu_\varepsilon)$ is smooth in a neighborhood of $x_0$ and $\rho_\varepsilon \geq \varepsilon_1 \on{vol}$. By compactness of $M$, iterating a finite number of times this argument gives the desired path.
\end{proof}

Next, we imply the equivalence between both  formulations above and a particular UOT problem introduced in Section \ref{SecDualRegularity}.

\subsubsection{Kantorovich formulation and dual formulation}
As in \cite{GeneralizedOT2} the application of Fenchel-Rockafellar Duality Theorem %on the cone 
gives the dual formulation of $\on{WFR}$. 
This is summarized in the following proposition.

\begin{proposition}[Dual formulation of $\on{WFR}$]\label{Th:Duality}
The following equality holds on $(M,g)$, 
\begin{equation*}
\on{WFR}^2(\rho_0,\rho_1) =
 \sup_{(\phi,\psi)\in C(M)^2  }  \int_M \phi(x) \ud \rho_0(x) + \int_M \psi(y) \ud \rho_1(y) \\
\end{equation*}
subject to the constraint 
\begin{equation*}
\begin{cases}\label{MultiplicativeConstraint}
 \phi(x) \leq 1\,, \quad \psi(y) \leq 1\, , \\
(1-\phi(x))(1-\psi(y)) \geq \cos^2\left(d(x,y)\wedge {(\pi / 2)}\right)\,,
\end{cases} 
\end{equation*}
for every  $(x,y)\in M^2$.
Setting $z_0=-\log(1-\phi)$ and $z_1=-\log(1-\psi)$ a reformulation of this linear optimization problem is
\begin{equation}\label{Eq:Reformulation}
\on{WFR}^2(\rho_0,\rho_1) =
 \sup_{(z_0,z_1)\in C(M)^2  }  \int_M \left(1-e^{-z_0(x)}\right) \ud \rho_0(x) + \int_M \left(1-e^{-z_1(y)}\right) \ud \rho_1(y) 
\end{equation}
subject to 
\begin{equation}\label{Eq:Constraint}
z_0(x)+ z_1(y) \leq -\log\left(\cos^2\left(d(x,y)\wedge {(\pi / 2)}\right)\right)\,,
\end{equation}
for every $(x,y)\in M^2$.
\end{proposition} 
Interestingly this last formulation is exactly the dual formulation of $\on{UOT}$ defined in Proposition \ref{prop:formulationduale} for $c(x,y)=-\log\left(\cos^2\left(d(x,y)\wedge {(\pi / 2)}\right)\right)$ and dual entropy functions $\FF_0^*(x) =\FF_1^*(x)= \FF^*(x) = e^{x} - 1$. 
The existence of Lipschitz solutions to the dual problem is proved in Corollary~\ref{Co:exilip},  under the admissibility condition on the measures, see Section \ref{SecExistenceLipschitz}.
Without these assumptions, the existence of potentials can be proved in a less regular space of functions, see \cite[Section 6.2]{LMS}.

\begin{proposition}[Kantorovich formulation of $\on{WFR}$] The following equality holds true on $M$
\begin{multline}\label{Eq:KLFormulation}
\on{WFR}^2(\rho_0,\rho_1) = \inf_{\gamma \in \mathcal{M}_+(M^2)}\on{KL}(p^1_*\gamma,\rho_0) + \on{KL}(p^2_* \gamma,\rho_1) \\ - \int_{M^2} \log(\cos^2(d(x,y) \wedge {(\pi / 2)})) \ud \gamma(x,y)\,.  
\end{multline}
\end{proposition}

\subsection{The formal Riemannian submersion and Monge formulation of $\on{WFR}$}\label{sec:mongeformulation} %
OT supports an interesting geometric framework.
Indeed, the push-forward action of the diffeomorphisms group on the space of densities is a (formal) Riemannian submersion to the space of densities endowed with the Wasserstein metric, see \cite{khesin2008geometry,DelanoeGeometryOT} for more details. 
This structure also exists in the case of UOT, as already explained in \cite{GALLOUET20184199}. We briefly recall it hereafter.

\par
Recall that a Riemannian submersion is a submersion $\pi$ between two Riemannian manifolds $M$ and $N$, such that $\d \pi$ is an isometry from the orthogonal of its kernel onto its range. An important property of Riemannian submersion is that every geodesic on $N$ can be lifted (called horizontal lift) to a unique geodesic on $M$ (having the same length), up to the choice of a basepoint in $M$.
In the following, the roles of $M$ and $N$ are taken by $
\Diff(M)$, the group of diffeomorphisms of $M$ and $\Dens_p(M)$ the space of probability densities on $M$.
We choose a reference volume form $\rho_0$ on $M$ and define
\begin{align*}& \pi_0: \Diff(M) \to \Dens_p(M) \\
&\pi_0(\varphi) = \varphi_* \rho_0
\end{align*} 
which is a (formal) Riemannian submersion of the metric $L^2(M;\rho_0)$ on  $ \Diff(M)$ to the Wasserstein $W_2$ metric on $\Dens_p(M)$. 
Using the horizontal lift property of geodesics mentioned above, the Benamou and Brenier dynamic formulation \cite{benamou2000computational} can be rewritten on the group $\Diff(M)$ as the Monge problem,
\begin{equation*}
W_2(\rho_0, \rho_1)^2 = \inf_{\varphi \in \Diff(M)} \left\{ \int_M d_{M}^2 (\varphi(x) , x)  \, \rho_0(x) \, \d \!\on{vol}(x) \, : \, \varphi_*\rho_0 = \rho_1 \right\}\,.
\end{equation*}

In the unbalanced case, the group $\Diff(M)$ is replaced with the semidirect product of groups between $\Diff(M)$ and the space of positive functions on $M$ which is a group under pointwise multiplication.
It is not a direct product but a semidirect one, where the composition law is defined such that the map $\pi_1$ given by 
\begin{align*}& \pi_1: \left( \Diff(M) \ltimes C(M,\R_{>0}) \right) \times \Dens(M) \mapsto \Dens(M) \\
&\pi_1\left((\varphi,\lambda) , \rho \right) \eqdef\varphi_* (\lambda \rho)
\end{align*}
is a left-action of the group $\Diff(M)  \ltimes C(M,\R_{>0})$ on the space of densities.
Similarly to the optimal transport case, this action is actually a Riemannian submersion between $L^2(M,\cone(M);\rho_0)$ and $\Dens(M)$ endowed with the $\on{WFR}$ metric.
Note that the $L^2$ metric is defined by a density (the initial density) on $M$ and the cone metric on $\cone(M)$ defined in Section~\ref{secsemicouplings} (see \cite{freed1989} for more details), namely
\begin{equation*}
\mathcal{H}_{(x,m)}(\d x,\d m) = a^2 m g_x + b^2 \frac{\d m^2}{m}\,\cdot
\end{equation*}
Up to the change of variable $m = r^2$, we find that the metric can be rewritten as 
\begin{equation}\label{EqConeMetric}
h_{(x,r)}(\d x,\d r) = a^2 r^2 g_x + 4b^2 \d r^2\,,
\end{equation}
which is called a cone metric\footnote{It is interesting to check that other Riemannian metrics on the cone can be chosen provided they are two-homogeneous in the radial variable. Some of the results of this article carry over such cases.}, see equation \eqref{eqmetricacono}. Since it is a classical formulation of this metric, we adopt this change of variable in the rest of the paper.
In particular, the action is changed into
\begin{align*}& \pi_{\rho}: \left( \Diff(M) \ltimes C(M,\R_{>0}) \right) \times \Dens(M) \to \Dens(M)\nonumber \\
&\pi_{\rho}\left((\varphi,\lambda) , \rho \right) \eqdef  \varphi_* (\lambda^2 \rho)\,,
\end{align*}
and the metric on $\cone(M)$ is the cone metric \eqref{EqConeMetric}. % We now adopt the notation $\cone(M)$ for the $M \times \R_{>0}$ equipped with the cone metric.
As done in \cite{GALLOUET20184199} we can identify this semidirect product of groups with the automorphism group of the cone $\cone(M)$ (since it has a multiplicative group structure in the $\R_{>0}$ component).
Thus, to shorten the notations, we use $\on{Aut}(\cone(M))$ instead of $\Diff(M) \ltimes C(M,\R_{>0})$.
We now state the (formal) Riemannian submersion result obtained in \cite{GALLOUET20184199}.

\begin{proposition}\label{prop:subriem}
Let $\rho_0 \in \Dens(M)$ be a positive density and $\pi$ be the map
\begin{align}& \pi:  \on{Aut}(\cone(M))  \to \Dens(M)\nonumber \\
&\pi(\varphi,\lambda) =  \varphi_* (\lambda^2 \rho_0)\,.\label{eq:submersion}
\end{align}
Then, $\pi$ is a Riemannian submersion between $\on{Aut}(\cone(M))$ endowed with the metric $L^2(M,\cone(M);\rho_0)$ and $\Dens(M)$ with the $\on{WFR}$ metric.
\end{proposition}
For details about the proof, we refer the reader to \cite{GALLOUET20184199}.
This proposition can be used to deduce a static or Monge formulation of the variational problem.
%\begin{equation}\label{Eq:MongeFormulation}
%\WF(\rho_0, \rho_1) = \inf_{(\varphi,\lambda) } \left\{ \| (\varphi,\lambda) - (\Id,1)  \|_{L^2(\rho_0)} \, : \, \varphi_* (\lambda^2 \rho_0 )= \rho_1 \right\}\,,
%\end{equation}

\begin{definition}
\label{Mongeformulation} Let $(\rho_0,\rho_1)\in \mathcal{M}_+(M^2)$.
The {\it Monge formulation of $\on{WFR}$} is given by
\begin{eqnarray}\label{Eq:MongeFormulation}
\on{M-WFR}^2(\rho_0, \rho_1) &=&\inf_{(\varphi,\lambda) } \left\{ \int_{M}\d^2_{\cone(M)}\left((x,1),(\varphi(x),\lambda(x))\right) d\rho_0(x) \, : \, \varphi_*(\lambda^2 \rho_0 )= \rho_1 \right\}\,,\\
&=&\inf_{(\varphi,\lambda) } \left\{ \d_{\mbox{Aut}(\cone(M))}^2\left( (\mbox{Id},1),(\varphi,\lambda)\right)\, : \, \varphi_*(\lambda^2 \rho_0 )= \rho_1 \right\}\,\nonumber
\end{eqnarray}
where the infimum is taken over $(\varphi,\lambda)\in\Diff(M) \ltimes C(M,\R_{>0})$ and $(\on{Id},1)$ denotes the identity  in $\mbox{Aut}(\cone(M))$.
\end{definition}
This Monge formulation extends to more general divergences and costs. Indeed, one can formulate
\begin{equation*}\on{M-UOT}^2(\rho_0, \rho_1) =\inf_{(\varphi,\lambda) } \left\{ \int_{M}\mathcal{D}_{\cone(M)}\left((x,1),(\varphi(x),\lambda(x))\right)^2 d\rho_0(x) \, : \, \varphi_*(\lambda^2 \rho_0 )= \rho_1 \right\}\,,
\end{equation*}
where \begin{equation}\label{EqConeFormulation}
\mathcal{D}_{\cone(M)}((x,r),(y,s))^2 = \inf_{z \in \R_{>0}} \left(r^2F_0(z/r^2) + s^2F_1(z/s^2) + c(x,y)z\right)\,.
\end{equation}
Let us remark that the quantity\footnote{Note that with respect to the first section we made the slight change of variable with the square root to remain consistent with the definition of the group action.} $\mathcal{D}_{\cone(M)}$ is not necessarily a power of a distance on the cone but it is the case in the three following situations. When $\FF_0 = \FF_1$ is the relative entropy, and $c(x,y) = -\log(\cos(d(x,y)\wedge\frac \pi 2)^2)$ the function $\mathcal{D}_{\cone(M)}((x,r),(y,s))$ is almost the distance on the cone but not exactly\footnote{For the cone distance, the minimum is taken with $\pi$ rather than $\pi/2$, this difference is explained by the fact that at the level of the measures, the transformation can occur simultaneously for both Dirac masses.} since
$\mathcal{D}_{\cone(M)}((x,r),(y,s))^2 = r^2 + s^2 - 2rs \cos(d(x,y)\wedge\frac \pi 2)$.
In this case, he equality between the two seemingly different Monge formulations holds.
For $\FF_0 = \FF_1$ is the relative entropy and $c(x,y) = d(x,y)^2$ (Gaussian-Hellinger case), there holds $\mathcal{D}_{\cone(M)}((x,r),(y,s))^2 = r^2 + s^2 - 2rse^{-d(x,y)^2/2}$.
The last known case is for \emph{partial optimal transport} where the divergences are taken as the total variation of measures given by the entropy function $F(x) = |x - 1|$ and the cost is $c(x,y) =d^q(x,y)$, for $q\geq 1$. Then,  
$\mathcal{D}_{\cone(M)}((x,r),(y,s))^q = r + s - (\min(r,s))\min(0,2 - d(x,y)^q)$ gives a distance.

A consequence of the semi-couplings formulation is the relaxation inequality  $\on{M-WFR}^2(\rho_0, \rho_1)\geq \on{WFR}^2(\rho_0, \rho_1)$. Indeed, it is sufficient to consider, for any $\phi$,  $\gamma(x,y)=(\Id,\phi)_{\#}\rho_0$, $\gamma_0(x,y)=\gamma(x,y)$ and $\gamma_1(x,y)=\lambda^2(x)\gamma(x,y)$.
The converse inequality does not hold in general since in the case of unbalanced transport not only the particles can split but also they can reach the apex of the cone.

However under our admissibility condition on $(\rho_0, \rho_1)$  we prove that $\on{M-WFR}^2(\rho_0, \rho_1)= \on{WFR}^2(\rho_0, \rho_1)$ in Proposition \ref{Th:PolarFactorisationStep}.

\subsection{Kantorovich relaxation: the conic formulation}\label{SecConicFormulation}
  Introduced in \cite{LMS}, the following  important formulation can be interpreted as a natural Kantorovich relaxation of the Monge formulation. 
The idea is to consider the pair $(\varphi,\rho)$ as a stochastic object.
From a cost on the cone defined by minimization in \eqref{EqConeFormulation}, one defines the \emph{conic formulation} 
\begin{equation*}%\label{EqConicFormulation}
\operatorname{C-OT}(\rho_0,\rho_1) = \inf_{\gamma \in \tilde\Gamma} \int_{\cone(M)^2} \mathcal{D}_{\cone(M)}((x,r),(y,s))^2 d\gamma((x,r),(y,s))\,,
\end{equation*}
where $\tilde\Gamma$ denotes the set of positive Radon measures $\gamma$ on $\cone(M)^2$  such that 
\begin{equation*}
    \begin{cases}
         \rho_0(x) = \int_{\R} r^2 [p^1_{*}\gamma](x,dr)\,,\\
          \rho_1(y) = \int_{\R} s^2 [p^2_{*}\gamma](y,ds)\,.
    \end{cases}
\end{equation*}
The last conditions are moment constraints rather than marginal ones as is the case in standard OT. Moreover, this formulation does not require the plan to be a probability measure on the product space although it can also be restricted to the set of probability measures by action with dilations, see \cite{LMS,AndreaRelaxation}. In fact, formula \eqref{EqConeFormulation} is $2$-homogeneous so that the mass can always be rescaled pointwisely.
Last, the moment constraint is the natural relaxation of the action by pushforward and rescaling $\varphi_*(\lambda^2\rho_0)$. 
Note that from the numerical point of view, introducing this additional radial variable is costly, yet it is amenable to entropic regularization, see \cite{sturmfels2023toric}. The proof of equivalence with the formulations introduced above can be found in \cite{LMS}. For our purpose and to prepare the discussion of $c$-convex functions in Section \ref{SecCconvex}, we simply note that the dual solutions of this problem are also dual solutions of an OT problem; the optimal potentials take the form $(x,r)\mapsto r^2p(x)$ and $(y,s)\mapsto s^2q(y)$ for functions $p,q$ defined on $M$. These potentials are $2$-homogeneous functions in the radial variable by construction.

\subsection{Monge solution and polar factorization on the automorphism group}
\label{secMonge}

The geometric structure used to show Brenier's polar factorization theorem \cite{Brenier1991} in standard optimal transport relies on the Riemannian submersion and solution of Monge problem. Thanks to results given in Section~\ref{sec:mongeformulation} and after finding a solution to the Monge problem $\on{M-WFR}$ we generalize in this section polar factorization to the unbalanced framework.

\subsubsection{Monge solution of $\,\on{WFR}$}

To show the existence of a solution to Monge problem \eqref{Eq:MongeFormulation} we
start by solving $\on{WFR} (\rho_0,\rho_1)$, in the dual form \eqref{Eq:Reformulation}, \eqref{Eq:Constraint} and we provide geometric properties of such solution (see Proposition~\ref{prop:approx}).
To prove Proposition \ref{prop:approx} there are two different arguments:  one is based on results in Section~\ref{SecDualRegularity} and the existence of Lipschitz potentials (Proof 1);  the other one mimics the standard case of optimal transport with minor adaptions due to the cost (Proof 2).
This latter approach leads to a pair of approximately differentiable potentials.
For completeness, we give both proofs.  For simplicity, introduce the notation $d_{\pi/2}(x,y)=d(x,y)\wedge {(\pi / 2)}$.
\begin{lemma} \label{subdiff}
For every $y \in M $, the function $x\mapsto g(x)= \cos^2\left(d_{\pi/2}(x,y)\right)$ is sub-differentiable.
\end{lemma}
\begin{proof}
The function  $ d(\cdot,y)$ is super-differentiable see the proof of \cite[Proposition 6]{McCann2001} for instance. Therefore, $d_{{\pi}/{2}}(\cdot,y)$ is also super-differentiable and 
the function $g$ is sub-differentiable as the combination of a decreasing $C^1$ function and the super-differentiable function  $ x\mapsto d_{{\pi}/{2}}(x,y)$, see \cite[Lemma 5]{McCann2001}. 
\end{proof}
Recall that, for a cost $c$ such that  $y\mapsto \nabla_x c(x,y)$ is injective on its domain of definition ({\it twist condition}), the $c$-exponential map at a point $x$ is defined by the identity $\cexp_x(v)=y$ if and only if $v=-\nabla_xc(x,y)$. We denote by $\tilde \nabla z(x)$ the approximate differential of a function $z:M\to\R$ at a point $x$. 
\begin{proposition}[Brenier's variational solution of $\on{WFR}$-Monge-Amp\`ere]\label{prop:approx}
Let $(\rho_0$, $\rho_1)\in\mathcal{M}_+(M^2)$  and let $(z_0, z_1)$ be generalized optimal potentials for $\on{WFR}^2(\rho_0,\rho_1)$.  
Suppose that $(\rho_0,\rho_1)$ is admissible and $\rho_0$ is absolutely continuous w.r.t $\on{vol}$, then $z_0$ is $\rho_0$ a.e. unique and approximate differentiable on $Supp(\rho_0)$. The optimal plan $\gamma$ in  \eqref{Eq:KLFormulation} is unique, with marginals $\gamma_0=e^{-z_0}\rho_0$, $\gamma_1=e^{-z_1}\rho_1$ and concentrated on the graph of  
\begin{equation}\label{eq:varpphi}
\varphi(x)  =  \exp_x \left(-\arctan \left(\frac{\|  \tilde\nabla z_0( x)\|}{2}\right)   \frac{ \tilde\nabla z_0( x) }{\| \tilde\nabla z_0( x)\| }\right)=\cexp_x (-\tilde \nabla z_0(x))\,,
\end{equation}
that is $\varphi_*\gamma_0=\gamma_1$ and $\gamma = ( \on{Id}\times \varphi )_{*} \gamma_0$. Finally
\begin{equation}\label{eq:optimalite}
\on{WFR}^2(\rho_0,\rho_1)=  \int_M \left(1-e^{-z_0(x)}\right) \ud \rho_0(x) + \int_M \left(1-e^{-z_1(y)} \right)\ud \rho_1(y)\,.
\end{equation}
\end{proposition}

Note that $(z_0, z_1)$ may not be continuous as needed in \eqref{Eq:Reformulation} but \eqref{eq:optimalite} still holds true.
 We start by giving  a simple sketch of the proof following the results in Section~\ref{SecDualRegularity}. We then provide in proof 2 the  approximate differentiability property %(being more technical) 
 and we discuss the corresponding formulation of the Monge-Amp\`ere equation.
 
\begin{proof}[Proof 1.]
Corollary \ref{Co:exilip} gives a pair of Lipschitz potentials $(z_0, z_1)$ solution to the dual formulation \eqref{EqDualUOTProblem} of $\on{WFR}^2(\rho_0,\rho_1)$. 
By Lemma \ref{ThLinearization}, the pair $(z_0,z_1)$  is also a solution of a classical Optimal transport problem between $\gamma_0=e^{-z_0}\rho_0$, $\gamma_1=e^{-z_1}\rho_1$ for the cost $c(x,y) = -\log\left(\cos^2\left(d_{\pi/2}(x,y)\right)\right)$. The hypothesis on $\rho_0$ and classical optimal transport theory arguments give the existence of a map $ \varphi(x)=\cexp (-\tilde\nabla z_0(x)) $ solution of this OT problem. In particular $\varphi_*\gamma_0=\gamma_1$. 
\end{proof}

Following the strategy  in the proof of Theorem \ref{ThReduction}, since $ \varphi(x)=\cexp (-\tilde \nabla z_0(x)) $ solves a classical optimal transport problem,  higher regularity of $z_0$ gives higher regularity of marginals $\gamma_0$ and $\gamma_1$ and, in turn,  a bootstrap argument improves the regularity of $z_0$.

\begin{proof}[Proof 2. (Approximate differentiability)]
The proof is an adaptation of \cite[Theorem 6.7]{LMS} using arguments in \cite{McCann2001,villanioldnew}. In particular, we use the notation of \cite{LMS}. Let $(z_0, z_1)$ be a pair of generalized optimal potentials  for $\on{WFR}^2(\rho_0,\rho_1)$ and $\gamma$ an optimal coupling in formulation \eqref{Eq:KLFormulation}, see \cite[Theorem 6.3]{LMS}. We define the associated densities $\sigma_i = e^{-z_i}$, $i=0,1$. 
Since $\rho_0$ and $\rho_1$ are admissible \cite[Theorem 6.3,b]{LMS} implies $\spt (p^1_*(\gamma))=\spt \gamma_0=\spt\rho_0$ and $\spt(p^2_*(\gamma))=\spt\gamma_1=\spt\rho_1$. Therefore, there exist Borel sets $A_i\subset \spt\rho_i $ with $\rho_i(M\setminus A_i)=0  $ such that 
\begin{align}
 %& z_1(x)+z_2(y) \leq -\ln(\cos^2(d^2_{\frac{\pi}{2}}(x,y))  &  \on{in  }  A_1 \times A_2     \\
 \label{eq:sigmaineq}&  \sigma_0(x) \sigma_1(y) \geq \cos^2( d_{{\pi}/{2}}(x,y ) )  &  \on{in  }  A_0 \times A_1 \,,    \\
 \label{eq:sigmaeq} & \sigma_0(x) \sigma_1(y) = \cos^2( d_{{\pi}/{2}}(x,y ) )  &   \gamma\on{-a.e.   in  } A_0 \times A_1\,.
\end{align}
To construct the set of approximate differentiability, define  
\begin{equation*}
    A_{1,n} := \left \{  y\in M \, ;\, \sigma_1(y) \geq 1/n    \right \},
\end{equation*}
and consider the function 
\begin{equation*}
    s_{0,n}(x)= \sup_{y\in A_{1,n}} \frac{\cos^2( d_{{\pi}/{2}}(x,y ) )}{\sigma_{1}(y)}\,.
\end{equation*}

By construction, $s_{0,n} $ is bounded, Lipschitz and thus differentiable  $\mathrm{vol}$-a.e. on $M$.
By definition, we have  $ \sigma_0 \geq s_{0,n}$ and thus the sequence of sets $A_{0,n} :=\left\{  x \in M \,; \,  \sigma_0(x) = s_{0,n}(x)  \right\} $ is increasing. Because of  \eqref{eq:sigmaeq}, the set $ \bigcap_{n=1}^{\infty} (X\setminus A_{0,n})$ is $\rho_0$-negligible. 
Let  
\begin{equation*}A'_{0,n} = \left\{  x \in A_{0,n}  \,; \,  \lim\limits_{r\to 0} \frac{\on{vol}(B(x,r)\cap A_{0,n} )}{\on{vol}(B(x,r))}=1   \;\text{and } s_{0,n} \text{ is differentiable at } x\right\}
\end{equation*}
be the set of points of $A_{0,n}$ with $\on{vol}$-density 1 and where $s_{0,n}$ is differentiable. Then  $ \bigcap_{n=1}^{\infty} (X\setminus A'_{0,n})$ is  $\rho_0$-negligible as well. % and thus $\on{vol}$ negligible.
Let $(\bar x,\bar y)\in A'_{0,n}\times A_{1,n}$ be such that 
$$
s_{0,n}(\bar x) \sigma_{1}( \bar y)  = \cos^2( d_{{\pi}/{2}}( \bar x, \bar y) )=\sigma_{0}(\bar x) \sigma_{1}( \bar y)\,.
$$
Using \eqref{eq:sigmaineq},   for all $x \in A_1$, there holds
$$
\sigma_{1}(\bar y)  \geq \cos^2( d_{{\pi}/{2}}(x,\bar y) )/ s_{0,n}(x)\,.
$$
In particular, $\cos^2( d_{{\pi}/{2}}(x,\bar y) )/ s_{0,n}(x)$ achieves its maximum at $\bar x$, implying $0 \in \nabla_{\bar x }^{+} ( \cos^2( d_{{\pi}/{2}} (\cdot,\bar y) )/ s_{0,n}(\cdot))$ (that is, $0$ is a supergradient of $x\mapsto \cos^2( d_{{\pi}/{2}} (x,\bar y) )/ s_{0,n}(x))$ at $\bar x$). Since $s_{0,n}$ is differentiable at $\bar x$, the function $x\mapsto \cos^2( d_{{\pi}/{2}} (x,\bar y))$     is super-differentiable at $\bar x$. By Lemma \ref{subdiff}, it is also sub-differentiable and thus differentiable at $\bar x$. Therefore,

\begin{align}
0&= \frac{\nabla  \cos^2\left( d_{{\pi}/{2}} (\bar x,\bar y) ) \right)}{ s_{0,n}(\bar x)}  - \cos^2( d_{{\pi}/{2}} (\bar x ,\bar y))  \frac{\nabla s_{0,n}(\bar x)}{  s^2_{0,n}(\bar x)} \nonumber\\ 
&= -\frac{ \cos^2( d_{{\pi}/{2}} (\bar x,\bar y) )}{s_{0,n}(\bar x)}\left(2 \tan( d_{{\pi}/{2}} (\bar x ,\bar y)) \nabla\left( \sqrt{2} \sqrt {\frac12 d^2_{{\pi}/{2}} (\bar x,\bar y) }\right)  + \nabla \ln  s_{0,n} (\bar x)\right)\nonumber \\\label{eqA3}
&=-\sigma_1(\bar y)\left(\frac{2 \tan( d_{{\pi}/{2}} (\bar x ,\bar y))}{{d_{{\pi}/{2}} (\bar x ,\bar y)}}\nabla \left( {\frac12 d^2_{{\pi}/{2}} (\bar x,\bar y) }\right) +   \nabla \ln  s_{0,n} (\bar x) \right)\,.
\end{align}
Let $v\in T_{\bar x}M$ be the unique vector such that $\bar y = \exp_{\bar x}(v)$. Then, $v=-\nabla \left( {\frac12 d^2_{{\pi}/{2}} (\bar x,\bar y) }\right)$ and \eqref{eqA3} implies
\begin{equation}\label{eqnablaz}
    \tilde \nabla z_{0} (\bar x)= -\tilde \nabla \ln  \sigma_{0} (\bar x)= -\nabla \ln  s_{0,n} (\bar x) =- 2 \tan (\|v\|)   \frac{v }{\| v\| }\,. 
\end{equation}
Therefore, $\bar y$ is unique $\rho_1$ a.e. and given by 
\begin{equation*}
\bar y  = \exp_{\bar x}\left( v \right)=\exp_{\bar x}\left( -\arctan \left(\frac{\| \tilde \nabla z_0(\bar x)\|}{2}\right)   \frac{\tilde \nabla z_0(\bar x) }{\| \tilde \nabla z_0(\bar x)\| }\right)= \varphi(\bar x)\,.
\end{equation*}

As a consequence,  $\gamma$ is concentrated on the graph of $\varphi$ in particular $\gamma =\left( \on{Id}\times\varphi \right)_*\gamma_0 $ and $\varphi_*\gamma_0=\gamma_1$.
The strict convexity of $\on{KL}$ implies that the marginals $\gamma_0$ and $\gamma_1$ are unique \cite[Theorem 6.7]{LMS} thus $$
z_0=-\log(\sigma_0) =-\log\left(\frac{\ud \gamma_0}{ \ud \rho_0}\right)
$$ 
is unique $\rho_0$ a.e. and $\gamma$ is also unique. Note that we used the admissible condition to say that $\sigma_0$ is $\rho_0$ a.e. positive. In order to prove \eqref{eq:optimalite}, we start from \eqref{Eq:KLFormulation} and a direct computation yields 
\begin{align}\label{egalitepotentiel}
\on{WFR}^2(\rho_0,\rho_1)&= \on{KL}(\gamma_0,\rho_0) + \on{KL}(\gamma_1,\rho_1) + \int_{M^2} c(x,y) \ud \gamma(x,y) \\ 
 \nonumber &= \int_M  \log \left(e^{-z_0} \right) e^{-z_0} \ud \rho_0 + \int_M (1-e^{-z_0}) \ud \rho_0 + \int_M \log \left(e^{-z_1} \right)e^{-z_1} \ud  \rho_1 + \int_M (1-e^{-z_1})\ud \rho_1 \\
 \nonumber &+ \int_{M^2} c(x,\varphi(x)) \ud \gamma(x)\\
 \nonumber &= \int_M (1-e^{-z_0}) \ud \rho_0 + \int_M (1-e^{-z_1})\ud \rho_1 +  \int_M  \left(c(x,\varphi(x)) - z_0(x)- z_1(\varphi(x)) \right) \ud  \gamma_0(x)\\
 \nonumber &= \int_M (1-e^{-z_0}) \ud \rho_0 + \int_M (1-e^{-z_1}) \ud \rho_1.
\end{align}
\end{proof}

As a consequence of the underlying classical OT structure, the potential $z_0$ found in Proposition \ref{prop:approx} is a solution to a Monge-Amp\`ere equation with a right-hand side that also depends on the potential.  
Let us recall  how to derive the equation, assuming that the optimal potential $z_0$ is of class  $C^2$. For the sake of readability, we denote $z_0$ by $z$. Recall that, by definition,  $\cexp_x(v)= \left[(-\nabla_x c) (x,\cdot)\right]^{-1}(v)$, which gives 
\begin{equation*}
    \cexp_x(v)=\exp_{x}\left( -\arctan \left(\frac{\|v\|}{2}\right)   \frac{v }{\| v\| }\right), \end{equation*}
 for $c(x,y)=-\log(\cos^ 2(d_{\pi/2}(x,y))$.
By equation \eqref{eqnablaz} above, $\nabla z$ satisfies
\begin{equation}\label{MA2}
  \nabla z (x)-(\nabla_x c) (x,\varphi(x)) =0\,.
\end{equation}
Identity  \eqref{MA2} is equivalent to 
\begin{equation*}
\varphi(x)=\cexp (-\nabla z(x)),
\end{equation*}
by definition of $\cexp$.
Differentiating \eqref{MA2} with respect to $x$ and taking the determinant yields
\begin{equation}\label{MA3}
\det\left[ -\nabla^2 z(x) +(\nabla_{xx}^2 c)(x,\varphi(x))\right]= \left|\det \left[ (\nabla_{x,y} c)(x,\varphi(x)) \right] \right| \left| \det( \nabla \varphi)\right|\,.
\end{equation}
Notice that, by $c$-concavity of  $z$, the symmetric matrix $ -\nabla^2 z +(\nabla_{xx}^2 c)(x,\varphi(x))$ is  non-negative. Note that $\varphi_*\left((1 + \frac14 \| \nabla z\|^2) e^{-2z} \rho_0\right) = \rho_1$ (see the proof of Proposition \ref{Th:PolarFactorisationStep} below) or, equivalently, 
\begin{equation*}
\left| \det( \nabla \varphi)\right| = e^{-2z }\left( 1 +  \frac14 \| \nabla z \|^2  \right) \frac{f}{g\circ \varphi }\,,
\end{equation*}
where  $f$, respectively  $g$, are Radon-Nikodym densities of $\rho_0$, respectively $\rho_1$, with respect to  $\on{vol}$.
Together with \eqref{MA3}, we deduce the $\on{WFR}$-Monge-Amp\`ere equation 
\begin{multline}\label{MA4}
\det\left[ -\nabla^2 z(x) +(\nabla_{xx}^2 c)(x,\varphi(x))\right]\\= \left|\det \left[ (\nabla_{x,y} c)(x,\varphi(x)) \right] \right| e^{-2z(x) }\left( 1 +  \frac14 \| \nabla z(x) \|^2  \right) \frac{f(x)}{g\circ \varphi(x) }\,,
\end{multline}
where $\varphi$ is given by \eqref{eq:varpphi} and satisfies the second boundary value problem: $\varphi$ maps  $\spt\rho_0$ towards  $\spt\rho_1$. 
\begin{remark}\label{rem:mongeampere}
Another possibility is to write directly the Monge-Ampère equation satisfied by $\varphi$ as an optimal map pushing $\gamma_0$ to $\gamma_1$, that is,
\begin{equation*}
\det\left[ -\nabla^2 z(x) +(\nabla_{xx}^2 c)(x,\varphi(x))\right]= \left|\det \left[ (\nabla_{x,y} c)(x,\varphi(x)) \right] \right|  \frac{e^{-z_0(x) }\rho_0(x)}{e^{-z_1\left(\varphi(x)\right) }\rho_1\circ \varphi(x) }\,\cdot
\end{equation*}
Using $z_0(x)+z_1(\varphi(x))=c(x,\varphi(x)) $ and $1 + \frac14 \| \nabla z_0(x) \|^2 = e^{c(x,\varphi(x))}\,$ one recovers equation \eqref{MA4}.
\end{remark} 
\begin{remark}
Following Brenier \cite[Section~1.4]{Brenier1991} Proposition \ref{prop:approx} can be taken as a definition of variational solutions for the $\on{WFR}$-Monge-Amp\`ere equation \eqref{MA4} with second boundary value problem.
The question of regularity of such a solution is a consequence of Theorem~\ref{ThReduction} in  Section~\ref{SecDualRegularity}: it relies on
regularity of classical a classical OT with cost $c$ and, therefore, on the study of the Ma-Trudinger-Wang tensor associated to $c$ see  \cite{figallide}, \cite[Chapter~12]{villanioldnew}. Note that partial regularity results such as one given in \cite{PMIHES_2015__121__81_0} can also be deduced from Theorem~\ref{ThReduction}. 
\end{remark}

Thanks to Proposition \ref{prop:approx} we are now able to prove the existence, under some assumptions on the initial density, of a solution to the Monge problem $\on{M-WFR}$.
 
\begin{proposition}[Solution of the Monge problem  $\on{M-WFR}$ and equivalence to $\on{WFR}$]\label{Th:PolarFactorisationStep}
Let $(\rho_0,\rho_1)$ be admissible and such that $\rho_0$ is absolutely continuous w.r.t. $\on{vol}$.   Then, there exists a  $\rho_0$-a.e. unique  $c$-concave function  $z:M\to\R$, approximatively differentiable $\rho_0$-a.e., such that the associated unbalanced transport couple $(\varphi,\lambda)$ defined by
\begin{eqnarray}\label{defphi}
\varphi(x) &=& \exp_x\left(-\arctan\left(\frac{1}{2} \|\tilde \nabla z(x) \| \right) \frac{\tilde \nabla z(x)}{\|\tilde \nabla z(x)\|}\right)\,,\\
\lambda(x) &=& e^{-z(x)} \sqrt{1 + \frac14\| \tilde \nabla z(x) \|^2}\label{deflambdai}
\end{eqnarray}
is a solution of the Monge problem \eqref{Eq:MongeFormulation} and satisfies 
\begin{equation}\label{Eq:MAEquation}
\pi[(\varphi,\lambda),\rho_0]=\varphi_*\left(\lambda^2 \rho_0\right)=\varphi_*\left(\left(1 + \frac14 \|\tilde \nabla z\|^2\right) e^{-2z} \rho_0\right) = \rho_1\,.
\end{equation}
Moreover, $(\varphi,\lambda)$ is the unique $\rho_0$-a.e. unbalanced transport couple associated to a $c$-concave potential, also unique, such that $\pi[(\varphi,\lambda),\rho_0]= \rho_1 $. 
The potential $z$ 
is characterized by  
\begin{equation}\label{eq:optimalite2}
\on{M-WFR}^2(\rho_0,\rho_1)=\on{WFR}^2(\rho_0,\rho_1)=  \int_M( 1-e^{-z(x)}) \ud \rho_0(x) + \int_M (1-e^{-\hat z(y)} )\ud \rho_1(y)\,,
\end{equation}
where $\hat z(y)=\inf_{x\in\spt \rho_0}c(x,y)-z(x)$.
%and $=\on{WFR}^2(\rho_0,\rho_1)$.
\end{proposition}

Proposition~\ref{Th:PolarFactorisationStep} appeared first in a preprint by the first and third author in \cite[Proposition~16]{gallouet2016camassaholm}. Since then, related results were given with the aim of numerical application in \cite{sarrazin2023linearized} and for a more theoretical use in \cite{liero2023fine}.

\begin{proof} We start by proving existence of solution for Monge problem.
Let $(z_0, z_1)$ be optimal potentials for $\on{WFR}^2(\rho_0,\rho_1)$. Then $\hat z_0=z_1$. Set $\sigma_i(x)=e^{-z_i(x)},i=0,1$. From Proposition \ref{prop:approx}, we know that $ x \mapsto \varphi(x)=\exp^M_{ x}\left(-\arctan \left(\frac{\| \tilde \nabla z_0( x)\|}{2}\right)   \frac{\tilde \nabla z_0( x) }{\| \tilde \nabla z_0( x)\| }\right)$ is well defined $\rho_0$ a.e. and 
$\varphi_*(\gamma_0)=\gamma_1$ where $\gamma_i=\sigma_i\rho_i $, $i=0,1$. Therefore
\begin{align*}
\rho_1 &=\sigma_1^{-1} \gamma_1=\sigma_1^{-1} \varphi_*(\gamma_0) =\sigma_1^{-1} \varphi_*\left(\sigma_0 \rho_0\right) \\
 &=  \varphi_*\left( e^{-z_0} \sigma_1^{-1} \circ \varphi  \rho_0\right) = \varphi_*\left( e^{-z_0} e^{z_1\circ \varphi } \rho_0\right)= \varphi_*\left( e^{-z_0} e^{ c(\cdot,\varphi(\cdot))} e^{-z_0}   \rho_0\right) \\
 &=  \varphi_*\left( e^{-2z_0}\left( 1 +  \frac14 \| \tilde \nabla z_0 \|^2  \right) \rho_0 \right)= \varphi_*\left(\left(e^{-z_0} \sqrt{1 + \frac14 \|\tilde \nabla z_0 \|^2} \right)^2  \rho_0\right)\\
 & =\pi\left[ \left( \varphi ,e^{-z_0} \sqrt{1 + \frac14\|\tilde \nabla z_0 \|^2} \right),\rho_0\right].%=\pi_{\rho_0} \left( \varphi ,e^{z_0} \sqrt{1 + \| \nabla z_0(x) \|^2} \right),
\end{align*}
where we use that $\rho_0$ a.e.  $z_0(x)+z_1(\varphi(x))=c(x,\varphi(x)) $, $1 + \tan^2(x) = 1/\cos^2(x)$ and thus $1 + \frac14 \| \tilde \nabla z_0(x) \|^2 = e^{c(x,\varphi(x))}\,$. 
Equation \eqref{eq:optimalite}  is exactly \eqref{eq:optimalite2}.

To prove uniqueness, consider $z$ to be a $c$-concave function, such that $(\varphi,\lambda)$ are well defined through \eqref{defphi}, \eqref{deflambdai} and 
$\pi[(\varphi,\lambda),\rho_0]=\rho_1$. Then, we claim that $ \gamma = (\on{Id}\times \varphi)_* (e^{-z} \rho_0)$ is an optimal plan for  $\on{WFR}^2(\rho_0,\rho_1)$ in \eqref{Eq:KLFormulation}. Indeed, let us check that $\gamma$ satisfies the optimality conditions of \cite[Theorem 6.3(b)]{LMS}, namely that its marginal $\gamma_i$ is absolutely continuous w.r.t. $\rho_i$, $i=0,1$ and 
\begin{eqnarray*}
    e^{-z(x)-\hat z(y)}&\geq &\cos^2(d_{\pi/2}(x,y)),\quad~~\forall(x,y)\in M\times M,\\
    e^{-z(x)-\hat z(y)}&= &\cos^2(d_{\pi/2}(x,y)),\quad\gamma\textrm{-a.e.}~ (x,y)\in M\times M.
\end{eqnarray*} 
Indeed, by definition of $\varphi$, 
\begin{equation*}%\label{cara}
z(x)+z^c(\varphi(x))= c(x,\varphi(x))\,, 
\end{equation*}
holds $\rho_0$-a.e.,
  and therefore $\gamma_0=e^{-z} \rho_0$-a.e. As a consequence, $\gamma$ is concentrated on the set of equality for a pair $(z,z^c)$ of $c$-concave functions, that is, 
$(z,z^c)$ satisfies for all $(x,y) \in M\times M$, $ z(x)+z^c(y) \leq c(x,y)$ with equality $\gamma$-a.e. Moreover, for $\rho_0$ almost every 
$x\in M$ there holds
\begin{equation*}
   \lambda^2(x)=e^{-2z(x)} (1 + \frac14\| \tilde \nabla z(x) \|^2)= e^{-z(x)}  e^{z^c(\varphi(x))}\,, 
\end{equation*}
  whence
\begin{equation*}
    \rho_1= \varphi_*(\lambda^2\rho_0)=  \varphi_*(   e^{z^c(\varphi(x))} e^{-z(x)} \rho_0)=e^{z^c} \varphi_*(   \gamma_0)= e^{z^c} \gamma_1\,.
\end{equation*}
Thus $\gamma_1= e^{-z^c}\rho_1$ and $\gamma$ is optimal for $\on{WFR}^2(\rho_0,\rho_1)$.
As a consequence, the equality $\on{M-WFR}^2(\rho_0,\rho_1)=\on{WFR}^2(\rho_0,\rho_1)$ holds.
The computation \eqref{egalitepotentiel} yields \eqref{eq:optimalite2} and uniqueness of the generalized optimal potentials for $\on{WFR}^2(\rho_0,\rho_1)$ in Proposition \eqref{prop:approx} implies uniqueness of $(z,\varphi,\lambda)$.
\end{proof}
\subsubsection{Polar factorization}
We end this section by showing  a polar factorization theorem for the automorphism group of the cone $\on{Aut}(\cone(M))$. 
In the sequel, we denote by $\Mes(X,Y)$ the set of measurable maps from a metric space $X$ to a metric space $Y$.
\begin{definition}
The {\it generalized automorphism semigroup} of $\cone(M)$ is the set of measurable maps %(denoted by $\Mes$ below) $\left(\varphi,\lambda \right)$ from $M$ to $\cone(M)$
\begin{equation*}
\autgc (\cone(M)) = \Mes(M,M)\ltimes \Mes(M,\R_{>0})\,,
\end{equation*}
endowed with the semigroup law
\begin{equation*}
(\varphi_1,\lambda_1) \cdot (\varphi_2,\lambda_2) = (\varphi_1 \circ \varphi_2,  (\lambda_1 \circ \varphi_2)  \lambda_2 )\,.
\end{equation*}
The {\it stabilizer} of the volume measure  is defined as 
\begin{equation*}
\autgc_{\on{vol}} (\cone(M)) = \left\{ \left(s,\lambda \right) \in \autgc (\cone(M)) \,:\, \pi\left((s,\lambda),\on{vol}\right)=\on{vol}  \right\}\,,
\end{equation*}
where $\pi$ is the submersion defined in \eqref{eq:submersion}, see Proposition~\ref{prop:subriem}. 
By abuse of notation, any $\left(s,\lambda \right) \in \autgc_{\on{vol}} (\cone(M))$ will be denoted  $\left( s, \sqrt{\on{Jac}(s)} \right)$ meaning that for every continuous function $f \in C(M,\R)$
\begin{equation}\label{weakracine}
\int_M f(s(x))\sqrt{\on{Jac}(s)}^2 \ud\on{vol}(x) = \int_M f(x) \ud \on{vol}(x)\,.
\end{equation}
\end{definition}

We denote by $T_{(x,r)}\cone(M)\ni(v,t)\mapsto \exp^{\cone(M)}_{(x,r)}(v,t)$ the exponential map for the cone metric on  $\cone(M)$.
\begin{theorem}[Polar factorization]\label{Th:polardecompo}
Let $(\phi,\lambda) \in \autgc (\cone(M))$ be 
such that the measure $\rho_1:=\pi\left[(\phi,\lambda),\on{vol} \right]$ is absolutely continuous with respect to $\on{vol}$ and $(\on{vol},\rho_1)$ is admissible. 
Then $\on{M-UOT}^2(\on{vol},\rho_1)$ admits a unique minimizer, which is characterized by a $c$-concave function $z_0$. 
Moreover, there exists a unique measure preserving generalized automorphism $(s, \sqrt{\on{Jac}(s)}) \in \isog (\cone(M))$ such that $\on{vol}$-a.e. 
\begin{equation*}
(\phi(x),\lambda(x)) = \exp_{(x,1)}^{\cone(M)}\left(-\frac12 \tilde \nabla p_{z_0}(x), -p_{z_0}(x) \right) \circ (s(x),\sqrt{\on{Jac}(s)}(x))\,
\end{equation*}
or equivalently
\begin{equation*}
(\phi,\lambda) =\left(\varphi,  e^{-z_0} \sqrt{1 + \|\tilde \nabla z_0 \|^2} \right)   \cdot (s,\sqrt{\on{Jac}(s)})\,,
\end{equation*}
where $p_{z_0}=e^{z_0}-1$ and 
\begin{equation*}
\varphi(x) = \exp_x\left(-\arctan\left(\frac{1}{2} \|\tilde \nabla z_0(x) \| \right) \frac{\tilde \nabla z_0(x)}{\|\tilde \nabla z_0(x)\|}\right)\,.
\end{equation*}
Moreover $(s, \sqrt{\on{Jac}(s)})$ is the unique  $L^2(M,\cone(M);\rho_0)$  projection of $ (\phi,\lambda)$ onto $\isog (\cone(M))$.

\end{theorem}

\begin{proof}[Proof.]% of Theorem \ref{Th:polardecompo}]
Define $\rho_0:=\on{vol}$.
Let $(z_0,z_1)$ be a solution of $\on{WFR}^2(\rho_0,\rho_1)$ and $\gamma$ be an optimal unbalanced transport plan. By symmetry, $(z_1,z_0)$ is a solution of $\on{WFR}^2(\rho_1,\rho_0)$ and $\gamma^t$ is an optimal unbalanced transport plan. Let finally $(\varphi_0,\lambda_0)$ and $(\varphi_1,\lambda_1)$ be the two transport couples given by application of Proposition \ref{prop:approx}  to $(\rho_0,\rho_1)$ and $(\rho_1,\rho_0)$. We split the proof into four  steps. We  denote by $\on{dom}(f)$ the domain of definition of a function $f$.\\

{\bf Step 1: $ \varphi_0 $ and $\varphi_1 $ are inverse maps.}
On $U=\varphi_0^{-1}(\on{dom}{\tilde\nabla z_1})\cap \on{dom}(\tilde \nabla z_0)$ which has full $\gamma_0$ measure and, therefore, full $\rho_0$ measure (we use here the admissible condition to say that $\gamma_0$ and $\rho_0$ have the same support), we have
$$
z_0(x)+z_1(\varphi_0(x))=c(x,\varphi_0(x))
$$
and thus 
 $ \varphi_1 (\varphi_0(x))=x$.  Similarly, it holds $ \varphi_0 (\varphi_1(y))=y$ on $V=\varphi_1^{-1}(\on{dom}{\tilde \nabla z_0})\cap \on{dom}(\tilde \nabla z_1)$ which has full $\rho_1$ measure.\\
 
 {\bf Step 2: $ (\varphi_0,\lambda_0)$ and $ (\varphi_1,\lambda_1) $ are inverse in $ \autgc$.}
 From Step 1, identity  $ \varphi_0 (\varphi_1(y))=y$ holds  $\rho_1$-a.e.. Thus,  $\rho_1$ a.e.
 $$
 (\varphi_0,\lambda_0) \cdot (\varphi_1,\lambda_1) = (\varphi_0\circ \varphi_1 ,\lambda_0 \circ \varphi_1  \lambda_1) = (\on{Id}, (\lambda_0 \circ \varphi_1) \lambda_1) \,.
 $$
Moreover by \eqref{Eq:MAEquation} of Proposition \ref{Th:PolarFactorisationStep} applied twice
\begin{align*}
\pi \left[ (\varphi_0,\lambda_0) \cdot (\varphi_1,\lambda_1), \rho_1 \right]=  \pi \left[ (\varphi_0,\lambda_0),  \pi \left[(\varphi_1,\lambda_1), \rho_1  \right] \right]   =  \pi  \left[ (\varphi_0,\lambda_0),\rho_0 \right]= \rho_1\,.
\end{align*}
It implies that 
$$ \pi \left[  (\on{Id}, (\lambda_0 \circ \varphi_1) \lambda_1) , \rho_1 \right]= \pi \left[ (\varphi_0,\lambda_0) \cdot (\varphi_1,\lambda_1), \rho_1 \right]=  \rho_1 \,.$$
In other words, we have $\rho_1$ a.e. $ (\lambda_0 \circ \varphi_1) \lambda_1=1$ and $\rho_1$ a.e.
$$
 (\varphi_0,\lambda_0) \cdot (\varphi_1,\lambda_1) = (\on{Id},1) \,.
$$
  
 {\bf Step 3: polar factorization.}
 Let $(s,\lambda_s)= (\varphi_1,\lambda_1) \cdot (\phi,\lambda)  = (\varphi_1 \circ \phi ,\lambda_1\circ \phi  \lambda ) $. By construction, one has
 \begin{align*}
\pi \left[ (s,\lambda_s), \rho_0 \right]=  \pi \left[ (\varphi_1,\lambda_1) \cdot (\phi,\lambda), \rho_0 \right] =  \pi \left[ (\varphi_1,\lambda_1),  \pi \left[ (\phi,\lambda),\rho_0  \right] \right]   =  \pi  \left[ (\varphi_1,\lambda_1),\rho_1 \right]= \rho_0\,.
\end{align*}
Therefore, $(s,\lambda_s)$ belongs to $\isog$ and $\lambda_s=\sqrt{\on{Jac}(s)} $ holds in the weak sense \eqref{weakracine}. Thus 
$$
(\phi,\lambda) = (\on{Id},1)\cdot(\phi,\lambda) = (\varphi_0,\lambda_0) \cdot   (\varphi_1,\lambda_1) \cdot (\phi,\lambda) = (\varphi_0,\lambda_0) \cdot (s, \sqrt{\on{Jac}(s)} )\,.
$$
It proves the polar factorization. 

{\bf Step 4: Uniqueness.}
The pair of $c$-concave potentials $(z_0,z_1)$ is optimal for $\on{WFR(\rho_0,  \left[ (\varphi_0,\lambda_0),\rho_0 \right] )} =  \on{WFR(\rho_0,  \rho_1 )}$ and therefore by Proposition \ref{Th:PolarFactorisationStep}, $ z_i $ are unique $\rho_i$ a.e..
We deduce that the projection  $(s,\sqrt{\on{Jac}(s)})= (\varphi_1,\lambda_1) \cdot (\phi,\lambda) $ is also unique $\rho_0$ a.e.. Indeed the positivity of $\lambda$ implies that $\spt(\lambda^2\rho_0)=\spt\rho_0$, thus $\phi$ maps $ \spt\rho_0$ onto $\spt\rho_1$ and the uniqueness of $\varphi_1$ and $\lambda_1$, $\rho_1$ a.e.,
implies the uniqueness of $s$ and $\sqrt{\on{Jac}(s)}$, $\rho_0$ a.e..
To prove that   $(s,\sqrt{\on{Jac}(s)}) $ is the $L^2(M,\cone(M);\rho_0)$  projection of $ (\phi,\lambda)$ onto $\isog (\cone(M))$, we observe  
\begin{align*}
\inf_{(\sigma,\sqrt{\on{Jac}(\sigma)} ) \in \isog (\cone(M)) } & \int_M d^2_{\cone(M)} \left( (\phi,\lambda) , (\sigma,\sqrt{\on{Jac}(\sigma)} ) \right) \rho_0 \geq \on{WFR}^2(\rho_0, \rho_1) \\
& 
=  \int_M d^2_{\cone(M)} \big( (\varphi_0,\lambda_0), (\Id,1) \big) \rho_0 \\ 
& = \int_M d^2_{\cone(M)} \left( (\varphi_0,\lambda_0)\cdot (s,\sqrt{\on{Jac}(s)} ) , (s,\sqrt{\on{Jac}(s)} ) \right) \rho_0\\
& =  \int_M d^2_{\cone(M)} \left((\phi,\lambda) , (s,\sqrt{\on{Jac}(s)} ) \right) \rho_0\,,
\end{align*}
which gives the result.
\end{proof}

As in OT, it is possible to drop the absolute continuity assumption on $\rho_1$ and extend Theorem \ref{Th:polardecompo} to this setting. In such a case, one loses the uniqueness of the measure preserving generalized automorphism $(s, \sqrt{\on{Jac}(s)})$. 
Another extension is to project on the subset of $\autgc(\cone(M))$ 
\begin{equation*}
    \isogh (\cone(M)) = \left\{ \left(s,\lambda \right) \in \autgc(\cone(M)) \, \middle| \, \pi\left((s,\lambda),\rho_0\right)=\mu_0  \right\}\,,
\end{equation*}
in the spirit of \cite[Theorem 3.15]{VillaniTOT}. The proof is similar to the one given above.
Last, the linearization of this polar factorization leads to a Helmholtz decomposition for velocity vector fields. 
\begin{remark}[Polar decomposition for other cost functions]
As explained previously, Remark~\ref{rem:mongeampere}, Proposition~\ref{prop:approx}, Proposition~\ref{Th:PolarFactorisationStep} and Theorem~\ref{Th:polardecompo} are not limited to the case of $\on{WFR}$ metric and can be deduced for more general costs.
For instance, a similar analysis for the Gaussian-Hellinger case in $\R^d$ is even easier to compute. Indeed, 
 in this setting the optimal potential $z$ is semi-concave, thus $\varphi$ turns out to be  the gradient of a convex function
\begin{equation*}
\varphi(x) = x- \nabla z(x) ,
\end{equation*}
and
\begin{equation*}
\lambda(x) = e^{-z(x)+\frac{1}{4} \|  \nabla z(x) \|^2 }.
\end{equation*}
This formulation can be particularly adapted for statistical or numerical applications, which we leave for future work.
\end{remark}
\section{The Ma-Trudinger-Wang tensor in the $\on{WFR}$ case: relations between $c$-convex functions and $\ddcone$-convex functions}\label{SecCconvex}% 
In this section we investigate the link between  $c$-convex functions on the base space $M$ and $\mathrm{d}^2_{\cone(M)}$-convex functions on $\cone(M)$. As a consequence, we provide a relation between the MTW-tensor on $M$ for the cost $c$ and the MTW-tensor on $\cone(M)$ for the cost  $\mathrm{d}^2_{\cone(M)}$.

We prove two fundamental facts. Lemma~\ref{ThLemmaCconvexity} states that a function is $c$-convex on $M$ if and only if its (suitably defined) lift is $\ddcone$-convex on  $\cone(M)$. Lemma~\ref{lem:csegment} is concerned with explicit computations along $c$-segments.

Let us recall the definition of cost-convex functions. 
\begin{definition}\label{defconv}\cite[Definition 5.2]{villanioldnew}
 Let $X \times Y \subset M\times M$ be a subset and $c$ be a cost function on $X \times Y$. A function $f: X  \to \R \cup \{ +\infty \}$ is {\it $c$-convex} if it is not identically $+\infty$ and if there exists a function $g: Y \to \R \cup \{ \pm \infty \}$ such that, for every $x \in X$, $f(x)=g^{c}(x)$, where $g^c$ is the {\it $c$-conjugate} of $g$, i.e., 
 \begin{equation*}
     g^c(x) := \sup_{y \in Y}  g(y) - c(x,y) \,.
 \end{equation*} 

 The {\it  $c$-subdifferential} of $f$ at point $\bar x$, denoted by $\partial_c f(\bar x)$,  is the set of $y\in Y$ such that, for every $x\in X$,
\begin{equation*}
f(x)\geq f(\bar x)+c(\bar x,y)-c(x,y).
\end{equation*}
\end{definition}
By definition, $f$ is $c$-convex if and only if, for every $\bar x\in X$, the set $\partial_cf(\bar x)$ is not empty.

In the sequel we set $\cos_+(x,y):=\cos\left(d_{\pi/2}(x,y)\right)$ and we consider the cost $c(x,y) = -\log(\cos^2_+(x,y))$. The corresponding distance on the cone is given by
\begin{equation*}
\ddcone((x,r),(y,s))=r^2+t^2-2 r t \cos_+(x,y).
\end{equation*}
\begin{definition}
 Given a function $f:M\to\R$ we define the {\it lift} of $f$ to $\cone(M)$ as the function $F_f: \cone(M) \to \R$ 
 \begin{equation*}
     F_f(x,r)=r^2(e^{f(x)}-1).
 \end{equation*}  
\end{definition}
This definition naturally arises, recalling that formulation \eqref{Eq:Reformulation} of $\on{WFR}$ metric can be seen as a dual formulation on the cone.

\begin{lemma}\label{ThLemmaCconvexity}
Let $X\times Y\subset M\times M$ and $f:M\to \R$. Then $f$ is $c$-convex on $X\times Y$ if and only if $F_f$ is $\ddcone$-convex on $(X \times \R_+)\times (Y\times \R_+)$. 
In particular, given $(\bar x,\bar r)\in X\times\R_+$,  $y \in \partial_c f (\bar x)$ if and only if $\left(y,s\right)\in \partial_{\ddcone}F_f(\bar x,\bar r)$ where $s=\bar r \frac{e^{f(\bar x)}}{\cos_+(\bar x,y)}$.
Finally, $(F_f)^{\ddcone}=F_{f^c}$. 
\end{lemma}
The last statement in Lemma~\ref{ThLemmaCconvexity} says that the lift of the $c$-conjugate of $f$ coincides with the $\ddcone$-conjugate of the lift of $f$.
\begin{proof}
By Definition~\ref{defconv}, it is sufficient to prove the second statement.

The function $f$ is $c$-convex  on $ X \times Y $, if and only if for every $\bar x\in X$ the $c$-subdifferential of $f$ at $\bar x$ is not empty.  In particular, for every $\bar x \in X$ there exists $y\in  Y$ such that, for every  $x \in X$,
\begin{eqnarray*}%\label{fsousdiff}
    f(x)&\geq& f(\bar x)+ c(\bar x, y) -c(x,y)\\
    &=&f(\bar x)-\log(\cos^2_+(\bar x, y))+\log(\cos^2_+(x, y)),\nonumber
\end{eqnarray*}
or,  equivalently, for every  $x \in X$,
\begin{equation}\label{fsousdiff2}
 e^{f(x)-f(\bar x)}\frac{\cos^2_+(\bar x,y)}{\cos^2_+(x,y)} \geq 1.
\end{equation}
Let now $\bar r\in\R_+$. Then $ \left(y,s\right)\in \partial_{\ddcone}F_f(\bar x,\bar r)$ if and only if, for every $(x,r)\in X\times \R_+$,
the following inequality holds true  
\begin{equation}\label{Fsousdiff}
   r^2(e^{f(x)}-1)\geq \bar r^2(e^{f(\bar x)}-1) +\ddcone\left((\bar x,\bar r),(y,s))\right) -\ddcone\left((x,r),(y,s))\right).
\end{equation}
Using the definition of $\dcone$, \eqref{Fsousdiff} is equivalent to 
\begin{equation}\label{Fsousdiff2}
   r^2e^{f(x)}\geq \bar r^2e^{f(\bar x)}  -2s\bar r\cos_+(d(\bar x,y)) +2sr\cos_+(d(x,y)).
\end{equation}
Adding  $s^2\cos^2_+(x,y)e^{-f(x)}+s^2\cos^2_+(\bar x,y)e^{-f(\bar x)}$ to both sides of  \eqref{Fsousdiff2}, the inequality becomes

\begin{align}\nonumber
   &e^{f(x)}\left( r- s\cos_+( x,y)e^{-f(x)} \right)^2
   - e^{f(\bar x)}\left( \bar r- s\cos_+(\bar x,y)e^{-f(\bar x)} \right)^2\\ 
   &+ s^2 \cos^2_+(x,y)e^{-f(x)}\left(\frac{\cos^2_+(\bar x,y)}{\cos^2_+(x,y)}e^{f(x)-f(\bar x)} - 1\right)\geq 0\label{Fsousdiff3}
\end{align}
For $F_f$ to be $\ddcone$-convex, \eqref{Fsousdiff3} must be satisfied for every $(x,r)\in X\times \R_+$. When this is the case,  evaluating  \eqref{Fsousdiff3} at $x=\bar x$ implies that, for every $r\in\R_+$,
\begin{equation}\label{ausil}\left( r- s\cos_+(\bar x,y)e^{-f(\bar x)} \right)^2
   - \left( \bar r- s\cos_+(\bar x,y)e^{-f(\bar x)} \right)^2\geq 0.
   \end{equation}
For a given $\bar r \in \R_+$ \eqref{ausil} holds for every $r\in\R_+$ if and only if 
\begin{equation*}%\label{svaleur}
s=\bar r\frac{e^{f(\bar x)}}{\cos_+(\bar x,y)}.
\end{equation*} 
Thus, the (unique) value of $s$ has been identified and we now evaluate  \eqref{Fsousdiff3} at this value. Inequality \eqref{Fsousdiff3} holds for every $(x,r)\in X\times\R_+$ if and only if  
\begin{equation}\label{equality2}
   e^{f(x)}\left( r- s\cos_+(x,y)e^{-f(x)} \right)^2 + s^2 \cos^2_+( x,y)e^{-f(x)}\left(\frac{\cos^2_+(\bar x,y)}{\cos^2_+( x,y)}e^{f(x)-f(\bar x)} - 1\right)\geq 0.
\end{equation}
If \eqref{equality2} holds true for every $(x,r)\in X\times\R_+$, then evaluating at $r=s\cos_+( x,y)e^{-f(x)}$ we infer  that 
\begin{equation*}
\frac{\cos^2_+(\bar x,y)}{\cos^2_+( x,y)}e^{f(x)-f(\bar x)} - 1\geq 0
\end{equation*} 
must be satisfied for every  $x\in X$, that is to say \eqref{fsousdiff2}, i.e., $y\in \partial_cf(\bar x)$.

The other direction is obvious since in Formula~\eqref{equality2} the first term is a square and the second term is nonnegative due to \eqref{fsousdiff2}. 
The proof of $(F_f)^{\ddcone}=F_{f^c}$ is done similarly or can be seen as a consequence of the identification of the subdiffentials.
Remark that all the inequalities hold true if $f$ takes an infinite value.  
\end{proof}

The next lemma makes a link between the notions of $c$-segment on $M$  and $\ddcone$-segment on $\cone(M)$.
Let us recall the definition of  cost-segments on a manifold.
\begin{definition}\label{def:csegment}\cite[Definition 12.10]{villanioldnew}
    Let $c:M\times M\to\R$ be a cost, $\bar x\in M$, and consider the parameterized segment between $q_0,q_1 \in T_{\bar x}M$ given by  $[0,1]\ni t\mapsto q_t=(1-t)q_0+tq_1$.  
    The  {\it $c$-segment}, whenever it is defined, is given by the parameterized curve
\[
[0,1]\ni t\mapsto y_t:=-(\nabla_{x}c(\bar x, \cdot ))^{-1}q_t.
\]
In this case, we refer to $\bar x$ as the {\it base point} of the $c$-segment.
Recalling that, by definition, $\cexp_{\bar x}\left( v\right)=y$ if and only if $-\nabla_xc(\bar x, y)=v$, $c$-segments coincide with the image under $c$-exponential map of segments in the tangent space. 
In order to keep track of the endpoints and basepoint in $M$, we introduce the notation 
$t\mapsto [y_0,y_1]^c_{\bar x}(t)$ for the $c$-segment given by $\cexp_{\bar x}(q_t)$, where $y_i=\cexp_{\bar x}(q_i), i=0,1$. 
\end{definition}

\begin{lemma}[Link between cost-convex segment]\label{lem:csegment}
  Let  $t\mapsto y_t=[y_0,y_1]^c_{\bar x}(t)=\cexp_{\bar x}(q_t)$ be a $c$-segment on $M$, where $q_t=(1-t)q_0+t q_1\in T_{\bar x} M$. Then, for every $(\bar r, a_0)$ with $\bar r \in \R^*_+$, $a_0> - 2\bar r$ there exist $s_0,s_1 \in \R_+$  such that the curve $t\mapsto (y_t,s_t)\in\cone(M)$, with 
  \begin{equation}\label{equs}
      s_t =  \frac{2\bar r +a_0}{2\cos_+(\bar x,y_t)},
  \end{equation}
  coincides with the $\ddcone$-segment $t\mapsto [(y_0,s_0),(y_1,s_1)]^{\ddcone}_{\bar x, \bar r}(t)=\coneexp_{(\bar x,\bar r)}\left(p_t,a_t\right)$ where
  \begin{equation*}%\label{equp}
p_t= \left( \bar r^2  +\frac{a_0}{2}\bar r \right)q_t,\quad a_t\equiv a_0. 
\end{equation*}
Conversely, let $t\mapsto (y_t,s_t)\in\cone(M)$ be the  $\ddcone$-segment \begin{equation*}
    t\mapsto [(y_0,s_0),(y_1,s_1)]^{\ddcone}_{\bar x, \bar r}(t)= \coneexp_{(\bar x,\bar r)}\left(p_t,a_t\right),
    \end{equation*}
    where   $p_t=(1-t)p_0+tp_1\in T_{\bar x}M$ and $a_t\equiv a_0>-2\bar r$. 
Then $t\mapsto y_t\in M$ is the $c$-segment of $M$ $t\mapsto[y_0,y_1]^c_{\bar x}(t)=\cexp_{\bar x}(q_t)$  where
$q_t=\frac{2 p_t}{2\bar r^2+\bar r a_0}$. 
 \end{lemma}

\begin{proof}
Recall that $c(x,y)=-\log(\cos_+^2(x,y))$ and $\ddcone((x,r),(y,t))=r^2+t^2-2 r t \cos_+(x,y)$. Thus,
\begin{eqnarray*}
    &&-\nabla_xc(\bar x, z)=\partial_x[\log(\cos_+^2(\bar x,z))]=2 \frac{\partial_x[\cos_+(\bar x,z)]}{\cos_+(\bar x, z)},\\
    &&-\nabla_{(x,r)}\ddcone((\bar x, \bar r), (z, s))=2\left(\bar r s \partial_x [\cos_+(\bar x,z)], -\bar r +  s\cos_+(\bar x,z)\right).
\end{eqnarray*}
Therefore, 
a curve  $t\mapsto y_t\in M$ is the $c$-segment $[y_0,y_1]^c_{\bar x}(t)$ if and only if there exist $q_0,q_1\in T_{\bar x} M$ for which $y_t$ satisfies 
 \begin{equation}\label{csegment}
   (1-t)q_0+tq_1 =- \nabla_{x}c(\bar x, y_t ) = 2 \frac{\partial_x[\cos_+(\bar x,y_t)]}{\cos_+(\bar x, y_t)}, 
  \end{equation}
 where $y_i=\cexp_{\bar x}(q_i), i=0,1$ (for simplicity set $q_t=(1-t)q_0+tq_1$).
Similarly, a curve  $t\mapsto (y_t,s_t)\in\cone(M)$ is a $\ddcone$-segment if and only if there exist $a_0,a_1>0$ and $p_0,p_1\in T_{\bar x}M$ for which $(y_t,s_t)$ satisfies
\begin{equation}\label{conesegment}
    \begin{cases}
    -\partial_x \ddcone((\bar x,\bar r),(y_t,s_t)) = 2\bar rs_t \partial_x [\cos_+(\bar x,y_t)] = (1-t) p_0+t p_1\\
      -\partial_r \ddcone((\bar x, \bar r),(y_t,s_t)) = -2\bar r + 2s_t\cos_+(\bar x,y_t) = (1-t) a_0+ t a_1\,.
\end{cases}
    \end{equation}
Let $t\mapsto y_t=[y_0,y_1]^c_{\bar x}(t)=\cexp_{\bar x}(q_t)$, with $q_t=(1-t)q_0+t q_1\in T_{\bar x}M$. For simplicity, we look for solutions of \eqref{conesegment} where $a_0=a_1$. If $t\mapsto \cos_+(\bar x,y_t)$ does not vanish,  the second equation gives
\begin{equation*}%\label{auxx}
    s_t=\frac{a_0+2\bar r}{2 \cos_+(\bar x,y_t)},
\end{equation*}
which is strictly positive if $a_0>-2 \bar r$.
Plugging such choice of $s_t$ in the first equation of system \eqref{conesegment}, we look for $p_0,p_1\in T_{\bar x }M$ satisfying
\begin{equation*}
\bar r\frac{a_0+2\bar r}{\cos_+(\bar x,y_t)} \partial_x [\cos_+(\bar x,y_t)]=(1-t) p_0+t p_1.
\end{equation*}
Using \eqref{csegment}, the identity above reads
\begin{equation*}
    \frac{\bar r}{2}(a_0+2\bar r) q_t=(1-t) p_0+t p_1,
\end{equation*}
which is satisfied by the choice $p_i=\frac{\bar r}{2}(a_0+2\bar r)q_i$, $i=0,1$.

Conversely,  assume a $\ddcone$-segment is given by 
\begin{equation*}
    t\mapsto (y_t,s_t)=[(y_0,s_0),(y_1,s_1)]^{\ddcone}_{(\bar x, \bar r)}(t)= \coneexp_{(\bar x,\bar r)}\left(p_t,a_0\right),
\end{equation*}
 where $p_t=(1-t)p_0+t p_1$ and $a_0>-2\bar r$. Then the pair $(y_t,s_t)$ satisfies \eqref{conesegment}.
Define 
$q_t=\frac{ 2 p_t}{a_0\bar r + 2\bar r^2}$. Since $t\mapsto p_t$ is affine, so is $t\mapsto q_t$. Moreover by \eqref{conesegment}, $q_t$ satisfies 
\begin{equation*}
    q_t=\partial_x[\log(\cos_+^2(d(\bar x,y_t)))]=-\nabla_xc(\bar x,y_t).
\end{equation*} 
Therefore $t\mapsto y_t$ is a $c$-segment between endpoints $y_0,y_1$ and with base point $\bar x$.
\end{proof}
A direct consequence of the correspondence between $c$-segments and $\ddcone$-segments is the following.

\begin{corollary}[Link between cost convexity domains]
   Let $ Y\times \R_{>0} \subset \cone(M) $ be a $\ddcone$-convex set with respect to $(\bar x,\bar r )\in \cone(M)$. Then $ Y \subset M $ is a $c$-convex set with respect to $\bar x \in M$.
\end{corollary}
\begin{proof}
By definition see \cite[Definition~12.11]{villanioldnew}, $Y\times \R_+ \subset \cone(M) $ is $\ddcone$-convex set with respect to $(\bar x,\bar r )$ if every pair of points in  $Y\times \R_+$ can be joined by a $\ddcone$-segment with base point $(\bar x,\bar r)$. Take  $y_0,y_1\in Y$  such that there exists $q_0,q_1\in T_{\bar x}M$ with the property $y_i=\cexp_{\bar x}(q_i), i=0,1$. Let $a_0>-2 \bar r$ and define 
\begin{eqnarray*}
  p_i&=&\left(\bar r^2+\frac{a_0}{2}\bar r\right)q_i,~~i=0,1,\\
    s_i&=&\frac{2\bar r+a_0}{2\cos_+(\bar x, y_i)},~~i=0,1.
\end{eqnarray*}
By construction, the $\ddcone$-segment 
\begin{equation*}
    t\mapsto(y_t,s_t):=\coneexp_{(\bar x,\bar r)}((1-t)p_0+t p_1,a_0)
\end{equation*} is contained in $Y\times \R_+$ and has endpoints  $(y_0,s_0), (y_1,s_1)$. 
By Lemma~\ref{lem:csegment}, the curve $t\mapsto y_t$ coincides with the $c$-segment $\cexp_{\bar x}(q_t)$, where $q_t=(1-t)q_0+t q_1$.
\end{proof}

Assume $t\mapsto y_t=[y_0,y_1]^c_{\bar x}(t)\in M$ is a $c$-segment. The {\it support function} along $y_t$ is defined by  $h_x:[0,1]\to \R$ 
       \begin{equation*}
h_x(t) = c(\bar x,y_t) - c(x,y_t)\, .
\end{equation*} 
A synthetic formulation for the sign of the MTW  tensor is also given by the quasi or plain convexity of the support function along a cost-segment see for instance \cite[Section 1.5.b,c,d]{gallouetphd} for a summary of what was developed in \cite[Theorem 12.36, Proposition 12.15(i)]{villanioldnew} and \cite[Theorem 2.7]{kim2012towards}. 
The second crucial lemma of this section makes the link between support function along a $c$-segment and the support function along the lifted $\ddcone$-segment on $\cone(M)$. 
\begin{lemma}\label{Lem:csegmentequivalence} Assume $t\mapsto y_t=[y_0,y_1]^c_{\bar x}(t)\in M$ is a $c$-segment with 
support function $t\mapsto h_x(t)$.
Let $t\mapsto(y_t, s_t)=[(y_0,s_0),(y_1,s_1)]^{\ddcone}_{(\bar x, \bar r)}(t)$ be any $\ddcone$-segment associated to $[y_0,y_1]^c_{\bar x}(t)$ throughout Lemma \ref{lem:csegment} and denote by $H_{(x,r)}:[0,1]\to\R $ the corresponding support function, namely
   \begin{equation*}
H_{(x,r)}(t) = \ddcone((\bar x,\bar r ),(y_t,s_t) - \ddcone((x,r),(y_t,s_t)).
\end{equation*} 
Then $h_x$ and $H_{(x,r)}$ satisfy the following identity
\begin{equation*}
 h_x(t)= 2 \log \left( \frac{H_{(x,r)}(t)-\bar r^2 +r^2}{a_0\bar r+2\bar r^2}  +1 \right). 
\end{equation*}
\end{lemma}
Remark that for $h_x,H_{(x,r)}$ to be well defined the cost $c$ must satisfy some smoothness condition ensuring that $q_t$ is in the domain of definition of $\cexp$.

\begin{proof}
By definition  
\begin{align}\nonumber
   h_x(t) &= c(\bar x,y_t) - c(x,y_t)
   =-\log(\cos^2_+(\bar x,y_t))+\log(\cos^2_+(x,y_t))\\
   &= 2\log \left( \frac{\cos_+(x,y_t)}{\cos_+(\bar x,y_t)} \right) \,.\label{formulationc}
\end{align}
The support function on $\cone(M)$ is given by 
\begin{align*} 
H_{(x,r)}(t) &= \ddcone((\bar x,\bar r ),(y_t,s_t)) - \ddcone((x,r),(y_t,s_t))\\  
    &= \bar r^2 - r^2+2rs_t\cos_+(x,y_t) - 2\bar r s_t\cos_+(\bar x,y_t).
\end{align*}
Since $(y_t,s_t)$ is a $\ddcone$-segment, thanks to Lemma~\ref{lem:csegment}, it satisfies \eqref{equs}, whence  
\begin{equation*}
      2\bar rs_t =  \frac{a_0\bar r+2\bar r^2}{\cos_+(\bar x,y_t)}=  \frac{\bar a}{\cos_+(\bar x,y_t)},
\end{equation*}
where $\bar a=a_0\bar r+2\bar r^2>0$. Thus 
\begin{equation*}
H_{(x,r)}(t)-\bar r^2 +r^2=  \bar a\left(\frac{r\cos_+(d(x,y_t))}{\bar r\cos_+(d(\bar x,y_t))} - 1\right).
\end{equation*}
Finally
\begin{equation*}
\log \left( H_{(x,r)}(t)-\bar r^2 +r^2 + \bar a \right) - \log \bar a = \log\left( \frac{r\cos_+(d(x,y_t))}{\bar r\cos_+(d(\bar x,y_t))} \right),
\end{equation*}
which provides the statement thanks to   \eqref{formulationc}. 
\end{proof}

Thanks to Lemma~\ref{lem:csegment} and Lemma~\ref{Lem:csegmentequivalence} 
we provide an instance of an answer to the question raised in \cite[Example 3.9]{kim2007continuity}, which states 
\begin{quote}
    ``It remains interesting to find more general sufficient conditions on a Riemannian
manifold $(M, g)$ and function $f$ "..." for $f(d(x, y))$ to be strictly or weakly regular.''
\end{quote}
We prove hereafter the following sufficient condition: if the MTW tensor associated with $\ddcone$ satisfies the MTW weak condition on $\cone$, then so does the MTW tensor associated with $c(x,y)=-\log(\cos^2_+(d(x,y)))$ on $M$. 
Recall that the latter cost is associated with the Wasserstein-Fisher-Rao metric, see Corollary \ref{Co:exilip}.
It is worth mentioning that the argument we exhibit works for any cost $j(x,y)$ on the base manifold as long as $\nabla_x j(x,\cdot)$ is injective and continuous with inverse continuous on a small neighborhood of all $y_0\in M$.

We have two proofs of this result based on two different characterizations of MTW weak condition, one using quasi-convexity of  $c$-segments and the other one based on the so-called {\it assumption (C)}.

\begin{definition}\cite[p.288]{villanioldnew}
    A cost $c$ on $M \times M$ satisfies {\it Assumption (C)}~if for every $c$-convex function $f$ and for every $x \in M$ in its domain, the $c$-subdifferential $\partial_c f(x)$ is connected.
\end{definition}
\begin{lemma}\label{ThAssumptionC}
    If $\ddcone$ satisfies assumption (C) on $\cone(M)$ then $c$  satisfies assumption (C) on $M$. 
\end{lemma}
\begin{proof}
To prove assumption (C) for $c$, let $f:M\to\R$ be a $c$-convex function. (Note that  both on $M$ and on $\cone(M)$ connectedness is equivalent to path-connectedness.) Take $y_1,y_2\in\partial_cf(\bar x)$. Then, by Lemma~\ref{ThLemmaCconvexity}, $(y_i,s_i)\in\partial_{\ddcone}F_f(\bar x,\bar r)$, where $s_i=\frac{\bar r e^{f(\bar x)}}{\cos_+(\bar x,y_i)}$. By assumption (C) on $\ddcone$, $\partial_{\ddcone}F_f(\bar x,\bar r)$ is connected, hence there exists a continuous path $t\mapsto(y_t,s_t)\in \partial_{\ddcone}F_f(\bar x,\bar r)$, with endpoints $(y_0,s_0), (y_1,s_1)$. Again, by Lemma~\ref{ThLemmaCconvexity}, $(y_t,s_t)\in \partial_{\ddcone}F_f(\bar x,\bar r)$  if and only if 
\begin{equation*}
 y_t\in \partial_cf(\bar x),\quad   s_t=\frac{\bar r e^{f(\bar x)}}{\cos_+(\bar x,y_t)}.
\end{equation*}
In particular, $t\mapsto y_t$ is a continuous path in $\partial_cf(\bar x)$ between endpoints $y_0,y_1$, whence $\partial_cf(\bar x)$  is connected.
\end{proof}
We are in a position to provide the main theorem of this section. 
Recall that a cost $c$ satisfies the MTW weak condition if and only if, for every pair of points,  the MTW tensor associated with $c$ computed at any pair of $c$-orthogonal vectors is nonnegative (see also  Section \ref{sec:leeli}).
\begin{theorem}\label{ThmequiMTW}%[$MTW(0)$ on the cone \implies $ MTW(0)$ for $-\log(\cos^2_+(d(x,y)))$]
    If  $\ddcone$ on $\cone{(M)}$ satisfies the MTW weak condition, then the cost $c$ on $M$ satisfies the MTW weak condition.
\end{theorem}
We give two proofs of Theorem~\ref{ThmequiMTW}.
\begin{proof}[Proof 1.]
Recall that, under some convexity assumptions,  
\cite[Theorem 12.42]{villanioldnew} states that assumption (C) is equivalent to MTW weak condition. Both costs $\ddcone$ on $\cone (M)$ and $c$ on $M$ satisfy the requirements in \cite[Theorem 12.42]{villanioldnew}. Therefore, applying the result to $\ddcone$  we deduce that $\ddcone$ satisfies assumption (C). By Lemma~\ref{ThAssumptionC} also $c$ satisfies assumption (C) on $M$. Applying  \cite[Theorem 12.42]{villanioldnew} to $c$ we conclude that $c$ satisfies MTW weak condition.
\end{proof}
\begin{proof}[Proof 2.] 
By the results \cite[Proposition 12.15 (i), Theorem 12.42]{villanioldnew},
under the same convexity assumptions, 
MTW weak condition for a cost is equivalent to 
 the quasi-convexity of the support function along any cost-segment (see also  
 \cite[Theorem 2.7]{kim2012towards} where this remark was made for the first time and \cite[Section 1.5.b,c,d]{gallouetphd} for a summary).
  Assume $\ddcone$ satisfies MTW weak condition on $\cone(M)$. Then, for every $\ddcone$-segment $t\mapsto (y_t,s_t)$, the support function $H_{(x,r)}(t)=\ddcone((\bar x,\bar r), (y_t,s_t))-\ddcone((x,r), (y_t,s_t))$   is quasi-convex, i.e., 
 \begin{equation*}
     H_{(x,r)}(t)\leq \max \left(H_{(x,r)}(0),H_{(x,r)}(1) \right).
 \end{equation*}
 Let $t\mapsto y_t\in M$ be a $c$-segment, $x\in M$. By Lemma~\ref{lem:csegment}, $y_t$ is the projection on $M$ of a $\ddcone$-segment $t\mapsto (y_t,s_t)$. Moreover, by 
 Lemma~\ref{Lem:csegmentequivalence}, the support function $t\mapsto h_x(t)$ along $y_t$ and $t\mapsto H_{(x,r)}(t)$ are related by 
 \begin{equation*}
  h_x(t)= 2 \log \left( \frac{H_{(x,r)}(t)-\bar r^2 +r^2}{a_0\bar r+2\bar r^2}  +1 \right). 
\end{equation*}
By hypothesis, $H_{(x,r)}(t)$ is quasi-convex.  
Since $\log$ is an increasing function,  $\max \left(H_{(x,r)}(0),H_{(x,r)}(1) \right)=H_{(x,r)}(j)$ is equivalent to $\max \left(h_{x}(0),h_{x}(1) \right)=h_{x}(j)$. Since $a_0\bar r+2\bar r^2>0$, quasi-convexity of $t\mapsto H_{(x,r)}(t)$   implies quasi-convexity of $t\mapsto h_x(t)$.
We apply \cite[Proposition 12.15 (i), Theorem 12.42]{villanioldnew} to the cost $c$ and we conclude that   $c$ satisfies MTW weak condition. 
\end{proof}
Note that the argument to prove Theorem~\ref{ThmequiMTW} uses a synthetic strategy as illustrated in \cite[Chapter 26]{villanioldnew}, whereas the assumption is formulated in such a way that it could be tested via direct computations. Nevertheless, we tried unsuccessfully to implement a proof based only on symbolic computations.

\begin{remark} In order to mention some weaker results, it is useful to review this section, keeping \cite[Theorem~12.42]{villanioldnew} in mind.
Lemma \ref{Lem:csegmentequivalence} states an equivalence for $c$ to be regular on $M$ and for $\ddcone$ to be regular on a specific set of $\ddcone$-segments of $\cone(M)$. 
Whereas Lemma \ref{ThLemmaCconvexity} states an equivalence for $c$ to satisfy assumption (C) on $M$ and $\ddcone$ to satisfy assumption (C) on a specific class of $\ddcone$-convex functions of $\cone(M)$.
Both these conditions imply the weak MTW condition. 
Therefore assumption (C) or regularity for $\ddcone$ on a subdomain on these specific sets is enough to ensure that MTW weak condition for $c$ 
holds true on a subdomain on the base space. 
To prove Theorem \ref{ThmequiMTW} we also used the reverse results that assumption (C) or regularity for $\ddcone$ on a totally $\ddcone$-convex set D are implied by MTW weak condition for $\ddcone$.
\end{remark}

We end this section by applying the strategy used to obtain Theorem~\ref{ThmequiMTW} to derive a result on the cross-curvature tensor. 
Cross-curvature tensor is essentially the curvature tensor of the Kim-McCann metric without the orthogonality condition, see \cite{kim2012towards}. It is also referred to as $MTW(0,0)$ \cite[Section 1.5.b,c,d]{gallouetphd}. 
Thus, asking nonnegativity of the cross-curvature tensor is a stronger condition than asking for the MTW weak condition to hold true. 
Nevertheless, this condition enjoys a useful property, namely it is known to pass to Riemannian submersions and products of manifolds (as proved in \cite{kim2012towards}), i.e., nonnegativity of cross-curvature tensor is preserved, which may not be the case for the nonnegativity of the MTW tensor.

\begin{theorem}
    If the cross-curvature tensor for $\ddcone$ on  $\cone(M)$ is nonpositive, then so is the cross-curvature tensor for the cost $-\log(\cos^2_+(d(x,y)))$ on $M$. 
    \end{theorem}
\begin{proof}
A synthetic formulation for the sign of the cross-curvature tensor associated with a cost $c$ is given by the convexity/concavity of the support function along a c-segment 
\cite[Section 1.5.b,c,d]{gallouetphd} or \cite[Theorem 2.10]{kim2012towards}, namely:  convexity of support functions is equivalent to nonnegative cross-curvature tensor whereas concavity is equivalent to a nonpositive cross-curvature tensor.
Using Lemma \ref{Lem:csegmentequivalence} and the fact that $\log $ is a concave increasing function we get that 
$t\mapsto H_{(x,r)}(t)$ concave implies $t\mapsto h_x(t)$ is also concave. 
 \end{proof}
This result is not of direct interest for smoothness of unbalanced optimal transport since it requires nonnegativity of the cross-curvature tensor rather than nonpositivity.

As $\log$ is concave we cannot prove here a result similar to Theorem \ref{ThmequiMTW}, that would push the nonnegativity of cross-curvature tensor from the cone onto the base space.

 \section{Summary and future directions}
We have shown, not unsurprisingly, that regularity for unbalanced optimal transport can be reduced to the one of optimal transport through the linearization of the dual problem. 
Regularity, being a structural result in itself, is interesting outside analysis. 
For instance, regularity of optimal transport maps is key to mitigate the curse of dimension of statistical optimal transport as done in \cite{vacher2021dimension} and to obtain minimax rate of convergence for the statistical estimation of optimal potentials \cite{muzellec2021nearoptimal}. Our results should allow similar gains in the statistical estimation of unbalanced optimal transport. 
We focus on Wasserstein-Fisher-Rao metric since it is the natural length space associated with the problem. This particular case leads us to examine the MTW condition of the induced cost. 

Furthermore,  an open application of pur polar factorization can lead to new numerical scheme for the Camassa-Holm equation as done for incompressible Euler in \cite{gallouet2016lagrangian}.

Finally, we provide  an example of answer to a question (formulated in  \cite{kim2007continuity}) by showing that  weak MTW condition on the cone implies weak  MTW condition  on the base manifold. A surprisingly different result holds for cross-curvature, whose nonpositivity on the cone implies nonpositivity of the corresponding cost on the manifold.

\section*{Acknowledgements}
The second author acknowledges project “ConDiTransPDE”, Control, diffusion and transport problems in PDEs and applications,  project number E83C22001720005, funded by Universit\`a degli Studi di Roma ``Tor Vergata'', Rome Italy, Project PRIN2022 MUR Code 2022W58BJ5 "PDEs and optimal control methods in mean field games, population dynamics and multi-agent models" funded by Ministry of University and Research Italy and  the MUR Excellence Department Project MatMod@TOV awarded to the Department of Mathematics, University of Rome Tor Vergata, CUP E83C23000330006
\\
The last author was supported by the Bézout Labex (New Monge Problems), funded by ANR, reference ANR-10-LABX-58.

\bibliographystyle{plain}
\bibliography{articles}

\begin{thebibliography}{10}

\bibitem{UserGuideOT}
L.~Ambrosio and N.~Gigli.
\newblock {\em A user's guide to optimal transport}.
\newblock Lecture Notes in Mathematics. Springer Berlin Heidelberg, 2013.

\bibitem{bauer2021square}
Martin Bauer, Emmanuel Hartman, and Eric Klassen.
\newblock The square root normal field distance and unbalanced optimal
  transport.
\newblock {\em Appl. Math. Optim.}, 85(3):Paper No. 22, 40, 2022.

\bibitem{benamou2000computational}
J-D. Benamou and Y.~Brenier.
\newblock A computational fluid mechanics solution to the {M}onge-{K}antorovich
  mass transfer problem.
\newblock {\em Numerische Mathematik}, 84(3):375--393, 2000.

\bibitem{refId0}
{Benamou, Jean-David}.
\newblock Numerical resolution of an "unbalanced" mass transport problem.
\newblock {\em ESAIM: M2AN}, 37(5):851--868, 2003.

\bibitem{bredies2020superposition}
Kristian Bredies, Marcello Carioni, and Silvio Fanzon.
\newblock A superposition principle for the inhomogeneous continuity equation
  with {H}ellinger-{K}antorovich-regular coefficients.
\newblock {\em Comm. Partial Differential Equations}, 47(10):2023--2069, 2022.

\bibitem{Brenier1991}
Yann Brenier.
\newblock Polar factorization and monotone rearrangement of vector-valued
  functions.
\newblock {\em Comm. Pure Appl. Math.}, 44(4):375--417, 1991.

\bibitem{caf92}
L.~Caffarelli.
\newblock The regularity of mappings with a convex potential.
\newblock {\em J. Amer. Math. Soc.}, 5:99--104, 1992.

\bibitem{GeneralizedOT1}
L.~{Chizat}, B.~{Schmitzer}, G.~{Peyr{\'e}}, and F.-X. {Vialard}.
\newblock {An Interpolating Distance between Optimal Transport and Fisher-Rao}.
\newblock {\em Found. Comp. Math.}, 2016.

\bibitem{GeneralizedOT2}
Lenaic Chizat, Gabriel Peyre, Bernhard Schmitzer, and Fran\c{c}ois-Xavier
  Vialard.
\newblock Unbalanced optimal transport: dynamic and {K}antorovich formulations.
\newblock {\em J. Funct. Anal.}, 274(11):3090--3123, 2018.

\bibitem{chizat2016scaling}
Lenaic Chizat, Gabriel Peyr{\'e}, Bernhard Schmitzer, and Fran{\c{c}}ois-Xavier
  Vialard.
\newblock Scaling algorithms for unbalanced transport problems.
\newblock {\em Mathematics of Computation}, 2018.

\bibitem{figallide}
G.~De~Philippis and A.~Figalli.
\newblock The {M}onge-{A}mp\`ere equation and its link to optimal
  transportation.
\newblock {\em Bull. Amer. Math. Soc.}, 51:527--580, 2014.

\bibitem{philippis2013mongeampre}
Guido De~Philippis and Alessio Figalli.
\newblock The {M}onge--{A}mp{\`e}re equation and its link to optimal
  transportation.
\newblock {\em Bulletin of the American Mathematical Society}, 51(4):527--580,
  2014.

\bibitem{F-DP}
Guido De~Philippis and Alessio Figalli.
\newblock The {M}onge-{A}mp\`ere equation and its link to optimal
  transportation.
\newblock {\em Bull. Amer. Math. Soc. (N.S.)}, 51(4):527--580, 2014.

\bibitem{PMIHES_2015__121__81_0}
Guido De~Philippis and Alessio Figalli.
\newblock Partial regularity for optimal transport maps.
\newblock {\em Publications Math\'ematiques de l'IH\'ES}, 121:81--112, 2015.

\bibitem{deponti2020entropytransport}
Nicol\`o De~Ponti and Andrea Mondino.
\newblock Entropy-transport distances between unbalanced metric measure spaces.
\newblock {\em Probab. Theory Related Fields}, 184(1-2):159--208, 2022.

\bibitem{DelanoeGeometryOT}
Philippe Delano\"e.
\newblock Differential geometric heuristics for riemannian optimal mass
  transportation.
\newblock In Boris Kruglikov, Valentin Lychagin, and Eldar Straume, editors,
  {\em Differential Equations - Geometry, Symmetries and Integrability},
  volume~5 of {\em Abel Symposia}, pages 49--73. Springer Berlin Heidelberg,
  2009.

\bibitem{Feydy2017}
Jean Feydy, Benjamin CHARLIER, François-Xavier Vialard, and Gabriel Peyré.
\newblock Optimal transport for diffeomorphic registration.
\newblock In {\em MICCAI 2017}, Quebec, Canada, September 2017.

\bibitem{freed1989}
D.~S. Freed and D.~Groisser.
\newblock The basic geometry of the manifold of riemannian metrics and of its
  quotient by the diffeomorphism group.
\newblock {\em Michigan Math. J.}, 36(3):323--344, 1989.

\bibitem{gallouetphd}
Thomas Gallou{\"e}t.
\newblock {\em {Optimal Transport : Regularity and applications}}.
\newblock Theses, {Ecole normale sup{\'e}rieure de lyon - ENS LYON}, December
  2012.

\bibitem{gallouet2016lagrangian}
Thomas~O. Gallou\"{e}t and Quentin M\'{e}rigot.
\newblock A {L}agrangian scheme \`a la {B}renier for the incompressible {E}uler
  equations.
\newblock {\em Found. Comput. Math.}, 18(4):835--865, 2018.

\bibitem{AndreaRelaxation}
Thomas~O. Gallou\"{e}t, Andrea Natale, and Fran\c{c}ois-Xavier Vialard.
\newblock Generalized compressible flows and solutions of the {$H({\rm div})$}
  geodesic problem.
\newblock {\em Arch. Ration. Mech. Anal.}, 235(3):1707--1762, 2020.

\bibitem{GALLOUET20184199}
Thomas Gallouët and François-Xavier Vialard.
\newblock The camassa–holm equation as an incompressible euler equation: A
  geometric point of view.
\newblock {\em Journal of Differential Equations}, 264(7):4199 -- 4234, 2018.

\bibitem{gallouet2016camassaholm}
Thomas Gallouët and François-Xavier Vialard.
\newblock The camassa-holm equation as an incompressible euler equation: a
  geometric point of view, preprint arXiv:1609.04006v1, 2016.

\bibitem{GeorgiouPower}
Tryphon~T. Georgiou, Johan Karlsson, and Mir~Shahrouz Takyar.
\newblock Metrics for power spectra: An axiomatic approach.
\newblock {\em IEEE Transactions on Signal Processing}, 57(3):859--867, 2009.

\bibitem{khesin2008geometry}
B.~Khesin and R.~Wendt.
\newblock {\em The geometry of infinite-dimensional groups}, volume~51.
\newblock Springer Science \&amp; Business Media, 2008.

\bibitem{kim2007continuity}
Young-Heon Kim and Robert~J. McCann.
\newblock Continuity, curvature, and the general covariance of optimal
  transportation.
\newblock {\em J. Eur. Math. Soc.}, 12, 2007.

\bibitem{kim2012towards}
Young-Heon Kim and Robert~J. McCann.
\newblock Towards the smoothness of optimal maps on riemannian submersions and
  riemannian products (of round spheres in particular).
\newblock {\em Journal für die reine und angewandte Mathematik (Crelles
  Journal)}, 2012(664):1--27, 2012.

\bibitem{new2015kondratyev}
S.~Kondratyev, L.~Monsaingeon, and D.~Vorotnikov.
\newblock A new optimal trasnport distance on the space of finite {R}adon
  measures.
\newblock {\em Adv. Differential Equations}, 21(11):1117--1164, 2016.

\bibitem{kondratyev2019spherical}
Stanislav Kondratyev and Dmitry Vorotnikov.
\newblock Spherical {H}ellinger-{K}antorovich gradient flows.
\newblock {\em SIAM J. Math. Anal.}, 51(3):2053--2084, 2019.

\bibitem{kondratyev2019convex}
Stanislav Kondratyev and Dmitry Vorotnikov.
\newblock Convex {S}obolev inequalities related to unbalanced optimal
  transport.
\newblock {\em J. Differential Equations}, 268(7):3705--3724, 2020.

\bibitem{laschos2018geometric}
Vaios Laschos and Alexander Mielke.
\newblock Geometric properties of cones with applications on the
  {H}ellinger-{K}antorovich space, and a new distance on the space of
  probability measures.
\newblock {\em J. Funct. Anal.}, 276(11):3529--3576, 2019.

\bibitem{LeeLi}
Paul W.~Y. Lee and Jiayong Li.
\newblock New examples satisfying {M}a-{T}rudinger-{W}ang conditions.
\newblock {\em SIAM J. Math. Anal.}, 44(1):61--73, 2012.

\bibitem{LMS}
M.~{Liero}, A.~{Mielke}, and G.~{Savar{\'e}}.
\newblock {Optimal Entropy-Transport problems and a new Hellinger-Kantorovich
  distance between positive measures}.
\newblock {\em Inventiones Math.}, August 2018.

\bibitem{liero2023fine}
Matthias Liero, Alexander Mielke, and Giuseppe Savar\'{e}.
\newblock Fine properties of geodesics and geodesic {$\lambda$}-convexity for
  the {H}ellinger-{K}antorovich distance.
\newblock {\em Arch. Ration. Mech. Anal.}, 247(6):Paper No. 112, 73, 2023.

\bibitem{MTWregularity}
Xi-Nan Ma, Neil~S. Trudinger, and Xu-Jia Wang.
\newblock Regularity of potential functions of the optimal transportation
  problem.
\newblock {\em Arch. Ration. Mech. Anal.}, 177(2):151--183, 2005.

\bibitem{McCann2001}
R.J. McCann.
\newblock Polar factorization of maps on riemannian manifolds.
\newblock {\em Geometric {\&} Functional Analysis GAFA}, 11(3):589--608, 2001.

\bibitem{muzellec2021nearoptimal}
Boris Muzellec, Adrien Vacher, Francis Bach, François-Xavier Vialard, and
  Alessandro Rudi.
\newblock Near-optimal estimation of smooth transport maps with kernel
  sums-of-squares, preprint arXiv:2112.01907, 2021.

\bibitem{AndreaEmbedding}
Andrea Natale and Fran\c{c}ois-Xavier Vialard.
\newblock Embedding {C}amassa-{H}olm equations in incompressible {E}uler.
\newblock {\em J. Geom. Mech.}, 11(2):205--223, 2019.

\bibitem{rockafellar1971integrals}
R.T. Rockafellar.
\newblock Integrals which are convex functionals. {II}.
\newblock {\em Pacific Journal of Mathematics}, 39(2):439--469, 1971.

\bibitem{sarrazin2023linearized}
Clément Sarrazin and Bernhard Schmitzer.
\newblock Linearized optimal transport on manifolds, preprint arXiv:2303.13901,
  2023.

\bibitem{shen2021accurate}
Zhengyang Shen, Jean Feydy, Peirong Liu, Ariel~H Curiale, Ruben San
  Jose~Estepar, Raul San Jose~Estepar, and Marc Niethammer.
\newblock Accurate point cloud registration with robust optimal transport.
\newblock In M.~Ranzato, A.~Beygelzimer, Y.~Dauphin, P.S. Liang, and J.~Wortman
  Vaughan, editors, {\em Advances in Neural Information Processing Systems},
  volume~34, pages 5373--5389. Curran Associates, Inc., 2021.

\bibitem{sturmfels2023toric}
Bernd Sturmfels, Simon Telen, Fran\c{c}ois-Xavier Vialard, and Max von Renesse.
\newblock Toric geometry of entropic regularization.
\newblock {\em J. Symbolic Comput.}, 120:Paper No. 102221, 15, 2024.

\bibitem{sjourn2019sinkhorn}
Thibault Séjourné, Jean Feydy, François-Xavier Vialard, Alain Trouvé, and
  Gabriel Peyré.
\newblock Sinkhorn divergences for unbalanced optimal transport, arXiv
  1910.12958, 2019.

\bibitem{sejourne2020unbalanced}
Thibault Séjourné, François-Xavier Vialard, and Gabriel Peyré.
\newblock The unbalanced gromov wasserstein distance: Conic formulation and
  relaxation, preprint arXiv 2009.04266, 2020.

\bibitem{vacher2021dimension}
Adrien Vacher, Boris Muzellec, Alessandro Rudi, Francis Bach, and
  Francois-Xavier Vialard.
\newblock A dimension-free computational upper-bound for smooth optimal
  transport estimation, preprint arXiv:2101.05380, 2021.

\bibitem{VillaniTOT}
C\'{e}dric Villani.
\newblock {\em Topics in optimal transportation}, volume~58 of {\em Graduate
  Studies in Mathematics}.
\newblock American Mathematical Society, Providence, RI, 2003.

\bibitem{villanioldnew}
C\'{e}dric Villani.
\newblock {\em Optimal transport old and new}, volume 338 of {\em Grundlehren
  der mathematischen Wissenschaften [Fundamental Principles of Mathematical
  Sciences]}.
\newblock Springer-Verlag, Berlin, 2009.

\bibitem{pmlr-v129-wang20c}
Zihao Wang, Datong Zhou, Ming Yang, Yong Zhang, Chenglong Rao, and Hao Wu.
\newblock Robust document distance with wasserstein-fisher-rao metric.
\newblock In Sinno~Jialin Pan and Masashi Sugiyama, editors, {\em Proceedings
  of The 12th Asian Conference on Machine Learning}, volume 129 of {\em
  Proceedings of Machine Learning Research}, pages 721--736. PMLR, 18--20 Nov
  2020.

\end{thebibliography}

\end{document}